\documentclass[12pt,a4paper, oneside, bold,secthm,seceqn,amsthm,ussrhead,reqno]{article}
\usepackage[utf8]{inputenc}
\usepackage[english]{babel}
\usepackage[symbol]{footmisc}
\usepackage{amssymb,amsmath,amsthm,amsfonts,xcolor,enumerate,hyperref, comment,longtable, cleveref}
\usepackage{verbatim}
\usepackage{times}
\usepackage{cite}
\usepackage{pdflscape}
\usepackage{ulem}
\usepackage[mathcal]{euscript}
\usepackage{tikz}
\usepackage{hyperref}
 
\usepackage{cancel}
\usepackage{stmaryrd}
 
\usepackage{amsfonts}
\usepackage{amssymb}
\usepackage{times}
\usepackage{xcolor}
\usepackage{tkz-graph}
\usepackage{url}
\usepackage{float}
\usepackage{tasks}
\usepackage{array}

\usepackage{cite}
\usepackage{hyperref}

 \usepackage{fancyhdr} 
\usepackage{amsfonts}
\usepackage{amsmath}
\usepackage{eurosym}
\usepackage{geometry}

\usepackage{caption,booktabs}

\theoremstyle{plain}
\newtheorem{theorem}{Theorem}
\newtheorem{lemma}[theorem]{Lemma}
\newtheorem{proposition}[theorem]{Proposition}
\newtheorem{remark}[theorem]{Remark}
\newtheorem{example}[theorem]{Example}
\newtheorem{definition}[theorem]{Definition}
\newtheorem{corollary}[theorem]{Corollary}

\usetikzlibrary{arrows}

\usepackage{fouriernc}

\hypersetup{colorlinks=true,linkcolor=black,citecolor=black,urlcolor=blue}
\begin{document}
 \bigskip

\noindent{\Large
Nonassociative algebras of anti-biderivation-type}\footnote{
The authors thank Bauyrzhan Sartayev for some useful discussions.
The    work is supported by 
FCT   2023.08031.CEECIND, UID/00212/2025 and 2022.14950.CBM (Programa PESSOA).
}

 \bigskip

\begin{center}

 {\bf
Saïd Benayadi\footnote{Université de Lorraine, Laboratoire IECL, CNRS-UMR 7502, UFR MIM, 3 rue
Augustin Frenel, BP 45112, 57073 Metz Cedex 03, France; \ said.benayadi@univ-lorraine.fr}, 
Said Boulmane\footnote{Moulay Ismail University, Meknes, Morocco; \ s.boulmane@umi.ac.ma}   \&
   Ivan Kaygorodov\footnote{CMA-UBI, University of  Beira Interior, Covilh\~{a}, Portugal; \    kaygorodov.ivan@gmail.com}   
   
}

\end{center}

\ 

\noindent {\bf Abstract:}
{\it
The main purpose of this paper is to study the class of Jacobi-Jordan-admissible algebras whose product is
an anti-biderivation of the related Jacobi-Jordan algebra; we call them $\mathcal A{\rm BD}$-algebras.
We give characterizations of $\mathcal A{\rm BD}$-algebras and show that this class includes Jacobi-Jordan algebras,
symmetric anti-Leibniz algebras and anti-${\rm LR}$-algebras.
The derived algebra of an anti-${\rm LR}$-algebra is associative and commutative, and all identities of its commutator algebra
follow from the metabelian identity and one identity of degree $4$.
We also study $\mathcal A$flexible ${\mathcal A}{\rm BD}$-algebras and prove that a complex finite-dimensional Jacobi-Jordan algebra
admits a non-trivial post-Jacobi-Jordan structure if and only if its dimension is at least $3$.
Finally, we give the algebraic classification of complex $3$-dimensional $\mathcal A{\rm BD}$-algebras and show that
complex finite-dimensional simple ${\mathcal A}{\rm BD}$-algebras are anticommutative, while perfect ones need not be; the minimal dimension of perfect non-anticommutative ${\mathcal A}{\rm BD}$-algebras is $4$.\footnote{All  results
have been checked in the proof assistant Lean~4   with the Mathlib library.}
}

\bigskip


 \bigskip

 (secondary).

 
\tableofcontents

\section*{Introduction}
The Leibniz rule as an identity with two multiplications plays an important role in the definition of Poisson algebras.  
This identity can also be found in the definition of commutative post-Lie algebras \cite{Burd1} and biderivation-type $\big(${\rm BD}-type$\big)$ algebras \cite{BO24}.
Namely, we say that an algebra $({\rm A},\cdot)$ is a {\rm BD}-type algebra if it satisfies the following identities: 
\begin{longtable}{lclclcl}
$x\cdot [y, z]$&$=$&$[x\cdot y,  z]+ [y, x\cdot z]$ & \ and \ &
$[y, z]\cdot x$&$=$&$[y\cdot x,  z]+[y,z\cdot x],$
\end{longtable}
\noindent where $[x,y]=x\cdot y - y \cdot x.$
{\rm BD}-type  algebras are closely related to Lie and  Lie-admissible algebras.
Commutative analogs of Lie algebras are known as mock-Lie or Jacobi-Jordan algebras \cite{Z17,BBMM,BBMMJJ}. These algebras are commutative and satisfy the Jacobi identity.
They are closely related to anti-Leibniz and anti-Poisson-Jacobi-Jordan algebras.
The main aim of the present paper is to consider an analog of {\rm BD}-type algebras in the context of Jacobi-Jordan algebras. 
Namely, we work with algebras of anti-biderivation-type $\big({\mathcal A}{\rm BD}$-type$\big)$, 
which form an intermediate variety between Jacobi-Jordan and Jacobi-Jordan-admissible algebras.
Jacobi-Jordan-admissible algebras, i.e. algebras that under the symmetric product give
Jacobi-Jordan algebras, were characterized in \cite{BBMM} as $3$-power antiassociative algebras, and they are related to many unexpected varieties $\big($for example, 
each Koszul dual binary $\mathfrak{perm}$-algebra is Jacobi-Jordan-admissible \cite{DST}$\big)$.
Algebras of ${\mathcal A}{\rm BD}$-type are algebras satisfying the following identities 
\begin{longtable}{lclclcl}
$x\cdot (y\circ z)$&$=$&$-\ \big((x\cdot y)\circ z+y\circ(x\cdot z) \big)$ & \ and \ &
$(y\circ z)\cdot x$&$=$&$-\ \big((y\cdot x)\circ z+y\circ(z\cdot x)\big)$,
\end{longtable}
\noindent where $x \circ y=x\cdot y + y \cdot x.$
The notion of anti-biderivation plays the key role in the present definition.
It seems that the notion of anti-biderivation first appeared as a particular case of $\delta$-biderivations in \cite{YH}, which are a natural generalization of $\delta$-derivations that are of interest $\big($see references in \cite{K,Z}$\big)$.

\medskip 

In the present paper, 
we prove that each ${\mathcal A}{\rm BD}$-algebra is a nilalgebra with nilindex $3$ $\big($Proposition \ref{3nil}$\big)$ and satisfies some interesting identities 
$\big($Remark \ref{rem11}, Proposition \ref{pr13} and so on$\big)$.
Section \ref{subvar} is dedicated to studying some principal subvarieties:
anti-Leibniz algebras, antibicommutative algebras $\big($${\mathcal A}{\rm LR}$-algebras$\big)$ 
and ${\mathcal A}$flexible algebras. 
We prove that the derived algebra of an ${\mathcal A}{\rm LR}$-algebra is associative and commutative and that the commutator algebra of an ${\mathcal A}{\rm LR}$-algebra is metabelian (Theorem \ref{s4adm}); moreover, all identities of this commutator algebra follow from the metabelian identity and one more identity of degree $4$, and we describe them explicitly in every degree (Theorem \ref{alrid}). The operad ${\mathcal A}{\rm LR}$ is self-dual $\big($Proposition \ref{selfdual}$\big)$.
As a corollary, the operad $\mathcal{A}{\rm LR}$ satisfies the Dong property in the sense of Kolesnikov--Sartayev \cite{11}.
We describe antiassociative ${\mathcal A}{\rm BD}$-algebras $\big($Corollary \ref{antiass}$\big)$ and
 prove that each  ${\mathcal A}$flexible ${\mathcal A}{\rm BD}$-algebra is an
 ${\mathcal A}{\rm BD}$-extension of a suitable Jacobi-Jordan algebra $\big($Theorem \ref{flex}$\big)$. 
The next section is dedicated to the study of post-Jacobi-Jordan structures on Jacobi-Jordan algebras: 
we prove that a complex finite-dimensional Jacobi-Jordan algebra admits a non-trivial post-Jacobi-Jordan structure if and only if its dimension is at least $3$
$\big($Proposition \ref{52} and Theorem \ref{Pronew}$\big)$.
The last section is devoted to  some finite-dimensional examples of  ${\mathcal A}{\rm BD}$-algebras: in particular, we demonstrate the differences between the variety of ${\mathcal A}{\rm BD}$-algebras and the varieties of 
flexible algebras and nilalgebras. 
The paper ends with the answer to the question about the structure of simple and perfect ${\mathcal A}{\rm BD}$-algebras:
complex finite-dimensional simple ${\mathcal A}{\rm BD}$-algebras are anticommutative (Theorem \ref{simple}), but there are perfect non-anticommutative ${\mathcal A}{\rm BD}$-algebras, already in dimension $4$ (Theorem \ref{perfect}).

\section{Preliminaries}	
In general, we work over the complex field $\mathbb C$,
but some results hold over other fields as well. An algebra is a vector space ${\rm A}$  together with a bilinear multiplication $(x,y) \rightarrow x\cdot y$ from ${\rm A}\times {\rm A}$ to ${\rm A}$. Let $({\rm A},\cdot)$ be an algebra. The polarization process determines two other multiplications ``$\circ$'' and ``$[\cdot,\cdot]$'' on $({\rm A},\cdot)$ defined by 
$$x\circ y=x\cdot y+y\cdot x\;\;\;\;\;\;    \mbox{and}\;\;\; \;\;\;    [x,y]=x\cdot y-y\cdot x.$$ The multiplication $\circ$ is commutative and the second one, $[\cdot,\cdot]$, is anticommutative.
The associator $\big($resp., antiassociator$\big)$ in an algebra is the trilinear function
\begin{center}
    $(x,y,z): =(x\cdot y)\cdot z-x\cdot (y\cdot z)$ \  
    $\big($resp., ${\mathcal A}(x,y,z) :=(x\cdot y)\cdot z+x\cdot (y\cdot z) \big)$. 
\end{center}

\noindent
An algebra $({\rm A},\cdot)$ is called:
\begin{enumerate}[I.] 
	\item  Associative $\big($resp., antiassociative\footnote{About antiassociative algebras see \cite{FKK} and references therein.}$\big)$ if $(x,y,z) = 0$ $\big($resp., ${\mathcal A}(x,y,z) = 0$$\big)$;
 
	\item   Commutative $\big($resp., anticommutative$\big)$ if $[x, y]=0$ $\big($resp., $x\circ y=0$$\big)$;
	\item  Flexible $\big($resp., antiflexible$\big)$ if $(x,y,z)=-(z,y,x)$  $\big($resp., $(x,y,z)=(z,y,x)$$\big)$;
	\item $\mathcal{A}$flexible if ${\mathcal A}(x,y,z)={\mathcal A}(z,y,x)$.
\end{enumerate} The depolarization process recovers the multiplication ``$\cdot$'' from the commutative multiplication ``$\circ$'' and the anticommutative multiplication ``$[\cdot,\cdot]$'': 
\begin{longtable}{lcl}
$x\cdot y$&$=$&$\frac{1}{2}\big([x,y]+x\circ y\big)$.
\end{longtable} \noindent
We denote  ${\rm A}^{+}:=({\rm A},\circ)$ and  ${\rm A}^{-}:=({\rm A},[\cdot,\cdot])$.

\begin{definition}[see \cite{Burd}]
 An algebra $({\rm A},\cdot)$ is called a Jacobi-Jordan algebra if it is commutative and satisfies the Jacobi identity: $$x\cdot (y\cdot z)+y\cdot (z\cdot x)+z\cdot (x\cdot y)=0.$$

\end{definition}

\begin{proposition}[see \cite{Z17}]
Jacobi-Jordan algebras can be characterized at least in the following four 
equivalent ways:
$$
\frac{\text{commutative}}{\text{Jordan}} \>\text{ algebras }\>
\frac{\text{satisfying the Jacobi identity}}{\text{of nil index } 3} .
$$
\end{proposition}

\begin{definition}[see \cite{BBMM}]
	{  A Jacobi-Jordan-admissible algebra is an algebra $({\rm A},\cdot)$ whose plus algebra ${\rm A}^{+}$ is a Jacobi-Jordan algebra. 
		
	}
\end{definition}

\begin{example}[see \cite{OK}]
	Every antiassociative algebra is  Jacobi-Jordan-admissible.
\end{example}

\begin{proposition}[see \cite{BBMM}]\label{JJadm}
An algebra $({\rm A},\cdot)$ is a Jacobi-Jordan-admissible algebra if and only if it satisfies 
\begin{center}
${\mathcal A}(x,y,z)+{\mathcal A}(x,z,y)+
{\mathcal A}(y,x,z)+{\mathcal A}(y,z,x)+
{\mathcal A}(z,x,y)+{\mathcal A}(z,y,x) \ = \ 0,$
\end{center}
or, equivalently, it is a $3$-power antiassociative algebra, i.e. $(x \cdot x)\cdot  x\ =\ -\ x\cdot (x\cdot  x).$
\end{proposition}

\begin{definition}[see \cite{F95}]
	Let $({\rm A},\cdot)$ be an algebra. A linear map  $\varphi:{\rm A} \to {\rm A}$ is called an antiderivation\footnote{Recall that the notion of antiderivations plays an important role in the definition of mock-Lie ($=$ Jacobi-Jordan) algebras \cite{Z17} and (transposed) anti-Poisson algebras \cite{dP}; 
in the characterization of Lie algebras with identities \cite{F95} and so on.} if it satisfies
		\begin{longtable}{lcl}	
		$\varphi (x\cdot y)$&$=$&$-\big(\varphi(x)\cdot y + x\cdot \varphi(y) \big)$.
\end{longtable} 
\end{definition}

Let us denote by ${\mathcal A}\mathfrak{Der}({\rm A},\cdot)$ the space of antiderivations of the algebra $({\rm A},\cdot).$
Unlike the usual derivations, the space of antiderivations is not closed under the commutator product: the commutator of two antiderivations is a derivation, 
but the space of antiderivations has the structure of a Lie module over the algebra of derivations (under the commutator product, i.e. the commutator of one derivation and one antiderivation gives an antiderivation).

\begin{definition}
	{  An algebra $({\rm A},\cdot)$ is called an algebra of anti-biderivation-type (${\mathcal A}{\rm BD}$-algebra) if
		\begin{equation}\label{Def1}	
		x\cdot (y\circ z)=-\big((x\cdot y)\circ z+y\circ(x\cdot z) \big),
		\end{equation}
		\begin{equation}\label{Def2}
		(y\circ z)\cdot x=-\big((y\cdot x)\circ z+y\circ(z\cdot x)\big).
		\end{equation}
  This is equivalent to saying that, for any element $x$  of ${\rm A},$ 
   ${\rm L}_{x}$ and ${\rm R}_{x}$ are antiderivations of the algebra $({\rm A}, \circ)$, where ${\rm L}_{x}$ and ${\rm R}_{x}$ denote, respectively, the left and right multiplications by the element $x \in {\rm A}$ in the algebra $({\rm A},\cdot)$.  
				
	}
\end{definition}

\begin{lemma}\label{bdnov}
 Let $({\rm A}, \cdot)$ be an ${\mathcal A}{\rm BD}$-algebra and assume that $I,J$ are two-sided ideals of $({\rm A}, \cdot)$.
Then   $I \circ J$   is   a two-sided ideal of $({\rm A}, \cdot).$
\end{lemma}

 \begin{proof}
     The statement follows from \eqref{Def1} and \eqref{Def2}.
 \end{proof}
\begin{example}
	{ 
Every Jacobi-Jordan algebra is an ${\mathcal A}{\rm BD}$-algebra.	
				
	}
\end{example}

\begin{proposition}\label{ABDJJA}
	 
 Every ${\mathcal A}{\rm BD}$-algebra $({\rm A},\cdot)$ is a Jacobi-Jordan-admissible algebra. 
 \end{proposition}
\begin{proof}
     Indeed, adding the identities \eqref{Def1} and \eqref{Def2}, we get 
  $$x\circ (y\circ z) \ =\ - \big((x\circ y)\circ z +y\circ(x\circ z) \big).$$

Then the conclusion follows.

\end{proof}

\begin{proposition}\label{3nil}
 Every ${\mathcal A}{\rm BD}$-algebra $({\rm A},\cdot)$ is a nilalgebra with nilindex $3.$  In particular, each ${\mathcal A}{\rm BD}$-algebra is a power associative algebra.
\end{proposition}

\begin{proof}
    Taking $x=y=z$ in \eqref{Def1}, we have 
\begin{center}    
    $x\cdot (x \cdot x+x \cdot x)\ =\ 
    -\big ( x \cdot (x \cdot x)+(x \cdot x)\cdot x + (x \cdot x) \cdot x + x \cdot (x \cdot x) \big ),$  
\end{center}
i.e. $4 x \cdot (x \cdot x)\ =\ -\  2 (x \cdot x)  \cdot x$ and thanks to Propositions
\ref{JJadm} and \ref{ABDJJA}, $x^3=0.$ 
Hence, taking $x=t\cdot t$ and $z=y=t,$ we have 
\begin{center}
    $(t \cdot t) \cdot (t \circ  t) \ = \ - \big(t^3 \circ  t+ t \circ t^3 \big) \ = \ 0,$ i.e. $x^4=0.$
\end{center}
Then, it follows from \cite[Lemma 3]{Albert}  that  $({\rm A},\cdot)$  is power-associative.

\end{proof}

\begin{definition}[see \cite{DU}]
    All algebras of a certain type are said to form a Nielsen--Schreier
variety if every subalgebra of a free algebra is free.
\end{definition}

\begin{corollary}
    The variety of ${\mathcal A}{\rm BD}$-algebras is not 
a Nielsen--Schreier variety. 
\end{corollary}

\begin{proof}
    Let ${\mathcal A}$ be the one-generated free ${\mathcal A}{\rm BD}$-algebra generated by an element ${\mathfrak x}.$ Obviously, ${\mathfrak x}^2\neq0.$
    Hence the subalgebra generated by ${\mathfrak x}^2$ is not free, as $({\mathfrak x}^2)^2=0.$
\end{proof}

\begin{definition}[see \cite{SG}]
An ideal of an algebra $({\rm A}, \cdot)$ is called characteristic $\big($resp., anti-characteristic$\big)$ if it is invariant under every derivation $\big($resp., antiderivation$\big)$ of  $({\rm A}, \cdot)$.
\end{definition}
\begin{proposition}
	Let $({\rm A}, \cdot)$ be a ${\rm BD}$-algebra $\big($resp.,  ${\mathcal A}{\rm BD}$-algebra$\big)$. 
    Then, every characteristic $\big($resp., anti-characteristic$\big)$ ideal $I$ of ${\rm A}^{-}$ $\big($resp., ${\rm A}^{+}$$\big)$ is a two-sided ideal of the algebra $({\rm A}, \cdot).$
\end{proposition}
\begin{proof}
Since the left and right multiplications of $({\rm A}, \cdot)$ are derivations of ${\rm A}^{-}$ $\big($resp., antiderivations of ${\rm A}^{+}$$\big)$, we have
${\rm L}_{x}(I)+{\rm R}_{x}(I)\subset I$ for all $x \in {\rm A}$, i.e., $I$ is a two-sided ideal of the algebra $({\rm A}, \cdot).$   
\end{proof}	

\section{Characterizations  of ${\mathcal A}{\rm BD}$-algebras}


\begin{remark}\label{rem11}
	{ 
Let $({\rm A},\cdot)$ be an   algebra. It is easy to see that identities \eqref{Def1} and \eqref{Def2} are equivalent to: 
\begin{equation}\label{Def2.}
{\mathcal A}(x,y,z)+{\mathcal A}(x,z,y)\ =\ -z\cdot (x\cdot y)-y\cdot (x\cdot z), 
\end{equation}
\begin{equation}\label{Def3.}
{\mathcal A}(y,z,x)+{\mathcal A}(z,y,x)\ =\ -(y\cdot x)\cdot  z-(z\cdot x)\cdot  y.
\end{equation}

}
\end{remark}

\begin{proposition}
    Let $({\rm A},\cdot)$ be an associative  algebra. 
    Then $({\rm A},\cdot)$ is a Jacobi-Jordan-admissible algebra if and only if it satisfies $x^3\ =\ 0$
    and $({\rm A},\cdot)$ is an  ${\mathcal A}{\rm BD}$-algebra if and only if it satisfies $x^3\ =\ 0$ and $[x, y \circ z]\ =\ 0.$
    \end{proposition}

\begin{proof}
First, each associative algebra is Jordan-admissible.
It is easy to see that 
\begin{center}
    $4\ x\cdot x \cdot x \ =\  x \circ (x \circ x),$
\ hence $x \circ (x \circ x)\ =\ 0$ if and only if $x\cdot x \cdot x=0.$
\end{center}

Second, adding \eqref{Def2.} and \eqref{Def3.}, we have 
\begin{center}    $x\cdot  y\cdot  z+y\cdot  x\cdot z+x\cdot z\cdot y+z\cdot x\cdot y+z\cdot y\cdot x+y\cdot z\cdot x=0$, \end{center} which is the linearization of $x^3=0.$
    Subtracting \eqref{Def3.} from \eqref{Def2.}, we obtain the second required identity.
\end{proof}

\begin{proposition}\label{pr13}
    Let $({\rm A},\cdot)$ be a nilalgebra with nilindex $3$. 
    Then $({\rm A},\cdot)$ is an ${\mathcal A}{\rm BD}$-algebra if and only if it satisfies  
   \begin{equation}\label{x3abd}
{\rm L}_y \circ {\rm L}_z \ =\ {\rm R}_y \circ {\rm R}_z.
\end{equation}
    
    \end{proposition}

\begin{proof}
It is easy to see that relations \eqref{Def2.} and \eqref{Def3.} can be rewritten in the following way:
\begin{longtable}{lcl}
    $(x\cdot y)\cdot  z+(x\cdot  z)\cdot  y$&$-$&$\big( z\cdot (y\cdot x)+y\cdot (z\cdot x)\big)\ =$ \\ 
    &$-$&$\big( x\cdot (y\cdot z)+y\cdot (x\cdot z)+x\cdot (z\cdot y)+z\cdot (x\cdot y)+z\cdot (y\cdot x)+y\cdot (z\cdot x) \big)$,\\

    $(x\cdot y)\cdot  z+(x\cdot  z)\cdot  y$&$-$&$\big( z\cdot (y\cdot x)+y\cdot (z\cdot x)\big)\ =$ \\ 
    &$-$&$\big( (x\cdot y)\cdot z+(y\cdot x)\cdot z+(x\cdot z)\cdot y+(z\cdot x)\cdot y+(z\cdot y)\cdot x+(y\cdot z)\cdot x \big)$,\\
\end{longtable}
which gives our statement.
\end{proof}

\begin{definition} 
    An algebra  $(\rm{A}, \star)$ for fixed elements $\alpha,\beta \in {\mathbb C}$  is called the $(\alpha,\beta)$-mutation of $(\rm{A}, \cdot)$   
    if the new multiplication is given by  
    $x \star y =  \alpha \ x\cdot y+\beta\ y\cdot x.$
\end{definition}

It is known that each $(\alpha,\beta)$-mutation of a Jacobi-Jordan-admissible algebra is a Jacobi-Jordan-admissible algebra \cite{BBMM}. 
The next proposition states the same for ${\mathcal A}{\rm BD}$-algebras.

\begin{proposition}
    If $(\rm{A}, \cdot)$ is an ${\mathcal A}{\rm BD}$-algebra, then each $(\alpha,\beta)$-mutation of $(\rm{A}, \cdot)$ is also an ${\mathcal A}{\rm BD}$-algebra.
\end{proposition}

\begin{proof}
 As $(\rm{A}, \cdot)$ is a nilalgebra with nilindex $3,$ we have 
\begin{longtable}{lclcl}
$(x \star x) \star x $&$=$&$ (\alpha+\beta) \big(\alpha (x\cdot x)\cdot x +\beta x\cdot (x\cdot x)\big) $&$=$&$0.$
\end{longtable}
\noindent Similarly, $x \star (x \star x)=0.$ 
Hence, we have to check only the identity \eqref{x3abd} for $(\rm{A}, \star),$ which holds due to the following observations.

\begin{longtable}{lclclclcl}
$(x \star y )\star z + (x \star z) \star y$ & $=$&
$\alpha^2\  (x\cdot y)\cdot z + \alpha\beta\  (y \cdot x)\cdot z + \alpha\beta \ z \cdot (x \cdot y) +\beta^2\  z \cdot (y \cdot x) $ \\
&& $\quad \quad + \alpha^2\ (x\cdot z)\cdot y + \alpha\beta\ (z \cdot x)\cdot y+\alpha\beta \ y \cdot (x \cdot z)+ \beta^2 \ y \cdot (z \cdot x);$\\ 

$z \star (y \star x) + y \star (z \star x)$ & $=$ & 
$\alpha^2\  z \cdot (y \cdot x) + \alpha\beta \ z \cdot (x \cdot y)+ \alpha\beta \ (y \cdot x) \cdot z + \beta^2\  (x \cdot y ) \cdot z$\\

&&$\quad \quad +\alpha^2\ y \cdot (z\cdot x)+ \alpha\beta\  y \cdot (x\cdot z)+\alpha\beta\  (z \cdot x)\cdot y +\beta^2 \ (x\cdot z)\cdot y.$

\end{longtable}
\end{proof}

\begin{definition}[see \cite{BBMM}]
	{ 
  Let $({\rm A},\cdot)$ be an algebra. 
$({\rm A},\cdot)$ is called a left $\big($resp., right$\big)$ pre-Jacobi-Jordan algebra   if 
  	\begin{longtable}{lclc lcl}
  	${\mathcal A}(x, y, z) $&$=$&$-\ {\mathcal A}(y, x, z)$ & 
    $\big($resp., &
  	${\mathcal A}(x, y, z)$&$=$&$-\ {\mathcal A}(x, z, y)$\ $\big)$.
  	\end{longtable}\noindent
$({\rm A},\cdot)$ is called a symmetric   pre-Jacobi-Jordan   algebra   if  
 it is a left and right pre-Jacobi-Jordan algebra.   
}
\end{definition}	
\begin{proposition}
    
	{ 
		
Let $({\rm A},\cdot)$ be a symmetric pre-Jacobi-Jordan algebra. Then
  $({\rm A},\cdot)$ is an ${\mathcal A}{\rm BD}$-algebra if and only if
\begin{equation}\label{Remark3}
(x\cdot y)\cdot  z\ =\ -\ (z\cdot y)\cdot  x,
\end{equation} 
\begin{equation}\label{Remark2}
x\cdot (y\cdot z)\ =\ -\ z\cdot (y\cdot x).
\end{equation}
The identities above are equivalent, respectively, to  

\begin{longtable}{lclclcl}
${\rm R}_{z} {\rm L}_{x}$&$=$&$-\ {\rm R}_{x} {\rm L}_{z}$ &
\ and \ & 
${\rm L}_{x} {\rm R}_{z}$&$=$&$-\ {\rm L}_{z} {\rm R}_{x}.$ 
\end{longtable}
\noindent Moreover, each of the identities \eqref{Remark3} and \eqref{Remark2} alone is equivalent to the ${\mathcal A}{\rm BD}$ property of a symmetric pre-Jacobi-Jordan algebra $\big($see the proposition on left and right pre-Jacobi-Jordan ${\mathcal A}{\rm BD}$-algebras below$\big)$.
  	}

\end{proposition}

\begin{proposition}
	Let $({\rm A}, \cdot)$ be a symmetric pre-Jacobi-Jordan ${\mathcal A}{\rm BD}$-algebra. Then 
	\begin{enumerate}
		\item[{\rm (a)}] $({\rm A}^{+})^{2}\subseteq {\rm Ann}({\rm A}^{-});$
		\item[{\rm (b)}] ${\rm A}^+$ is a metabelian Jordan algebra.
		\end{enumerate}
\end{proposition}
\begin{proof}
 Let us consider $x,y,z,t \in {\rm A}$, then
\begin{enumerate}
    
\item[{\rm (a)}]  we have $[x\circ y,z]\ = \ (x\circ y)\cdot z - z\cdot(x\circ y). $
It follows, by identities \eqref{Remark3} and \eqref{Remark2}, that  
\begin{longtable}{lclclclcl}
$[x\circ y,z]$&$=$&$ {\mathcal A}(x,y,z) + {\mathcal A}(y,x,z)$&$=$&$ 0.$
\end{longtable}
We conclude that $({\rm A}^{+})^{2}\subseteq {\rm Ann}({\rm A}^{-}).$

\item[{\rm (b)}] on the other hand, we have

\begin{longtable}{lclclclcl}
$(z\cdot t)\cdot (x\circ y)$&$=$&$- \big((x\circ y)\cdot t\big)\cdot z$&$=$&$-\big(t\cdot (x\circ y)\big)\cdot z$&$=$&\\
&$=$&$\big(z\cdot (x\circ y)\big)\cdot t$&$=$&$\big((x\circ y)\cdot z\big)\cdot t$&$=$&$-(t\cdot z)\cdot (x\circ y).$
\end{longtable}
Then,  $0\ =\ 
(z\circ t)\cdot (x\circ y)$, and so $(z\circ t)\circ (x\circ y)\ =\ 0$. Thus,  
${\rm A}^+$ is $2$-step solvable, i.e. metabelian.

\end{enumerate}\end{proof}

\begin{proposition}
	Let $({\rm A}, \cdot)$ be a left pre-Jacobi-Jordan ${\mathcal A}{\rm BD}$-algebra. Then   
	\begin{enumerate}
		\item[{\rm (a)}] $(x\cdot y)\cdot (z\cdot t)\ =\ -\ (x\cdot z)\cdot (y\cdot t)\footnote{Let us note that the present identity gives the ``anti-version'' of medial algebras $\big($i.e., satisfying the identity $(xy)(zt)=(xz)(yt)\big)$ considered in \cite{T24}.};$
		\item[{\rm (b)}] $\big((z\cdot x)\circ y \big)\cdot x \ =\ \big((y\cdot x)\cdot x\big)\circ z.$ 
    \end{enumerate} 
    
\end{proposition}
\begin{proof}
 Let us consider $x,y,z,t \in {\rm A}$, then
\begin{enumerate}
    
\item[{\rm (a)}]  
  it follows  from the identity \eqref{Remark3} that 
  \begin{longtable}{lclclclcl}
  $ (x\cdot y)\cdot (z\cdot t) $&$=$&$- \big((z\cdot t)\cdot y\big)\cdot x$&$=$&$\big((y\cdot t)\cdot z\big)\cdot x$&$=$&$- \ (x\cdot z)\cdot (y\cdot t).$
  \end{longtable}
\item[{\rm (b)}]  Also from equation \eqref{Remark3}, we deduce that 
${\rm L}_{x^{2}}\ =\ -{\rm R}^{2}_{x}$. It follows that ${\rm R}^{2}_{x} \in {\mathcal A}\mathfrak{Der}({\rm A}, \circ),$ i.e., 
	${\rm R}^{2}_{x}(y\circ z)\ =\ -{\rm R}^{2}_{x}(y)\circ z-y\circ {\rm R}^{2}_{x}(z).$ Hence,

\begin{longtable}{rcl}
$\big((y\circ z)\cdot x\big)\cdot x$&$=$&$-\big((y\cdot x)\cdot x\big)\circ z-y\circ \big((z\cdot x)\cdot x\big) $\\
$-\big((y\cdot x)\circ z+y\circ (z\cdot x)\big)\cdot x $&$=$&$-\big((y\cdot x)\cdot x\big)\circ z-y\circ \big((z\cdot x)\cdot x\big)$ \\
	$ \big((y\cdot x)\cdot x\big)\circ z+ (y\cdot x)\circ (z\cdot x)+(y\cdot x)\circ (z\cdot x)+y\circ \big((z\cdot x)\cdot x\big)$
	&$=$&$-\big((y\cdot x)\cdot x\big)\circ z-y\circ \big((z\cdot x)\cdot x\big)$  \\
	$(y\cdot x)\circ (z\cdot x)$&$=$&$-y\circ \big((z\cdot x)\cdot x\big)-\big((y\cdot x)\cdot x\big)\circ z$\\
	 $ -(z\cdot x)\circ (y\cdot x)- \big((z\cdot x)\cdot x\big) \circ y$&$=$&$\big((y\cdot x)\cdot x\big)\circ z$ 	\\
$\big((z\cdot x)\circ y\big)\cdot x $&$=$&$\big((y\cdot x)\cdot x\big)\circ z.$
\end{longtable}


\end{enumerate}
\end{proof}

\begin{proposition}\label{pr1}
	{ 
Let $({\rm A},\cdot)$ be an algebra. The following assertions are equivalent:
\begin{enumerate}
\item[{\rm (a)}] $({\rm A},\cdot)$ is an ${\mathcal A}{\rm BD}$-algebra$;$

\item[{\rm (b)}] $({\rm A},\cdot)$ is a Jacobi-Jordan-admissible algebra and ${\rm L}_{x}\in {\mathcal A}\mathfrak{Der}({\rm A}^{+})$, for all $x\in {\rm A}$.

\item[{\rm (c)}] $({\rm A},\cdot)$ is a Jacobi-Jordan-admissible algebra and ${\rm R}_{x}\in {\mathcal A}\mathfrak{Der}({\rm A}^{+})$, for all $x\in {\rm A}$.
\end{enumerate}
}
\end{proposition}
\begin{proof} Let us consider the following implications.

\begin{enumerate}
    \item[{\rm (a)}$\Rightarrow${\rm (b)}] is trivial. 

    \item[{\rm (b)}$\Rightarrow${\rm (c)}] Since ${\rm A}^{+}$ is a Jacobi-Jordan algebra,  for all $x\in {\rm A}$ the map ${\rm L}_{x}^{\circ}: {\rm A}\rightarrow  {\rm A}$ defined by ${\rm L}_{x}^{\circ}(y)=x\circ y, \forall y\in {\rm A}$ belongs to ${\mathcal A}\mathfrak{Der}({\rm A}^{+})$. It follows that ${\rm R}_{x}={\rm L}_{x}^{\circ}-{\rm L}_{x}\in {\mathcal A}\mathfrak{Der}({\rm A}^{+})$.

     \item[{\rm (c)}$\Rightarrow${\rm (a)}] Assume that ${\rm R}_{x}\in {\mathcal A}\mathfrak{Der}({\rm A}^{+})$. Since ${\rm L}_{x}^{\circ}\in {\mathcal A}\mathfrak{Der}({\rm A}^{+})$, we have ${\rm L}_{x}={\rm L}_{x}^{\circ}-{\rm R}_{x}\in {\mathcal A}\mathfrak{Der}({\rm A}^{+})$. So $({\rm A},\cdot)$ is an ${\mathcal A}{\rm BD}$-algebra. 
		
\end{enumerate}
 
\end{proof}	

\begin{corollary}
	{ 
		Let $({\rm A},\cdot)$ be an algebra. Then   $({\rm A},\cdot)$ is an ${\mathcal A}{\rm BD}$-algebra if and only if  $({\rm A},\cdot)$ is a Jacobi-Jordan-admissible algebra and it satisfies  \eqref{Def2.} or \eqref{Def3.}.
	}
\end{corollary}
\begin{corollary}
	{ 
		Let $({\rm A},\cdot)$ be a Jacobi-Jordan-admissible algebra.
        Then $({\rm A},\cdot)$ is an ${\mathcal A}{\rm BD}$-algebra if and only if 
        it satisfies   \eqref{Def2.} or  \eqref{Def3.}.
	}
\end{corollary}
 
\begin{lemma}
    Let $({\rm A}, \cdot)$ be an algebra. The following assertions are equivalent:
    \begin{enumerate}
        \item[{\rm (a)}] ${\rm L}_x-{\rm R}_x\in {\mathcal A}\mathfrak{Der}({\rm A}, \circ);$
        \item[{\rm (b)}] $[x, y \circ z] = - \big( [x,y] \circ z+ y \circ [x,z] \big)$;
        \item[{\rm (c)}] $ {\mathcal A}(x,y,z)+{\mathcal A}(x,z,y)-{\mathcal A}(y,z,x)-{\mathcal A}(z,y,x) = 
        (z,x,y)+(y,x,z).$    
    \end{enumerate}
\end{lemma}

\begin{proposition}\label{pr2}
	{ 		
Let $({\rm A},\cdot)$ be a Jacobi-Jordan-admissible algebra. Then $({\rm A},\cdot)$ is an ${\mathcal A}{\rm BD}$-algebra if and only if ${\rm L}_{x}-{\rm R}_{x} \in {\mathcal A}\mathfrak{Der}({\rm A}, \circ)$, i.e.
\begin{equation}\label{Eq2}
 [x,y\circ z]=- \big([x,y]\circ z+y\circ[x,z] \big).
\end{equation}

}
\end{proposition}
\begin{proof}
	{\rm
If $({\rm A},\cdot)$ is an ${\mathcal A}{\rm BD}$-algebra, then ${\rm L}_{x}-{\rm R}_{x}\in {\mathcal A}\mathfrak{Der}({\rm A}, \circ)$. Conversely, since $({\rm A},\cdot)$ is a Jacobi-Jordan-admissible algebra, we have ${\rm L}_{x}^{\circ}\in {\mathcal A}\mathfrak{Der}({\rm A}, \circ)$. Moreover, it is easy to see that \begin{longtable}{lcl}${\rm R}_{x}$&$=$&$\frac{1}{2}\big({\rm L}_{x}^{\circ}-({\rm L}_{x} -{\rm R}_{x})\big)$,
\end{longtable} \noindent that means ${\rm R}_{x}\in {\mathcal A}\mathfrak{Der}({\rm A}, \circ)$. Then by Proposition \ref{pr1},  $({\rm A},\cdot)$ is an ${\mathcal A}{\rm BD}$-algebra.       

}
\end{proof}
\begin{remark}
	{ 
Relation \eqref{Eq2} can be interpreted as saying that for any $z\in {\rm A}$, the linear map $[\cdot ,z]: x\rightarrow [x,z]$ is an antiderivation of ${\rm A}^{+}$, which is equivalent to saying that the bilinear map $[\cdot ,\cdot ]: {\rm A}\times {\rm A}\rightarrow {\rm A}$ is a skew-symmetric anti-biderivation of ${\rm A}^{+}$. Recall that a bilinear map $\delta: {\rm A}\times {\rm A}\rightarrow {\rm A}$ is called a skew-symmetric anti-biderivation of an algebra $({\rm A},\cdot)$ if it is skew-symmetric and satisfies the following identity:
\begin{longtable}{rcl}
$\delta(x, y)$&$=$&$-\delta(y, x),$\\
$\delta(x, y\cdot z)$&$=$&$-\big(\delta(x, y)\cdot z+y\cdot  \delta(x, z) \big).
$\end{longtable}
 }
\end{remark}	

\begin{remark}
	{ 
Proposition \ref{pr2} shows that there is a one-to-one correspondence between the set of ${\mathcal A}{\rm BD}$-algebras and the set of Jacobi-Jordan algebras with skew-symmetric anti-biderivations.
}
\end{remark}
\begin{corollary}
	{ 	
		Let $({\rm A},\cdot)$ be an antiassociative algebra. Then $({\rm A},\cdot)$ is an ${\mathcal A}{\rm BD}$-algebra if and only if $({\rm A},\cdot)$ is flexible.
		
	}
\end{corollary}
\begin{proof}
	{\rm Note that the identity \eqref{Eq2} is equivalent to: 
		\begin{equation}\label{Eq4}		
		{\mathcal A}(x,y,z)+{\mathcal A}(x,z,y)-{\mathcal A}(y,z,x)-{\mathcal A}(z,y,x) \ =\ (z,x,y)+(y,x,z).
		\end{equation}
		Since $({\rm A},\cdot)$ is antiassociative,   \eqref{Eq4} becomes $(z,x,y)+(y,x,z) \ =\ 0$, i.e. $({\rm A},\cdot)$ is flexible. The converse is obtained by using that every antiassociative algebra is Jacobi-Jordan-admissible. 
	}
\end{proof}

It is known that pre-Jacobi-Jordan algebras are a particular case of Jacobi-Jordan-admissible algebras. Hence, we deduce the following.
\begin{proposition}
	{ 		
		A left $\big($resp., right$\big)$ pre-Jacobi-Jordan algebra $({\rm A},\cdot)$ is an ${\mathcal A}{\rm BD}$-algebra if and only if \eqref{Remark3} $\big($resp., \eqref{Remark2}$\big)$ holds.					
	}
\end{proposition}
\begin{proof}
	{\rm	
		This follows directly from Proposition \ref{pr2} and Equation \eqref{Eq4}. 
	}
\end{proof} 

By \cite{Z17}, the variety of antiassociative anticommutative algebras coincides with the variety of Koszul dual mock Lie (= Koszul dual Jacobi-Jordan) algebras. Hence, we have the following corollary. 
\begin{corollary}\label{dml}
	{ 		
		Every Koszul dual Jacobi-Jordan algebra is an ${\mathcal A}{\rm BD}$-algebra. 				
	}
\end{corollary}

\begin{remark}
In fact, every anticommutative algebra is an ${\mathcal A}{\rm BD}$-algebra: if $x\circ y=0$ for all $x,y,$ then both sides of \eqref{Def1} and \eqref{Def2} vanish.
\end{remark}


\begin{definition}[see \cite{dP}] 
    A generic anti-Poisson-Jacobi-Jordan algebra is a triple 
    $({\rm A},\circ,[\cdot,\cdot]),$ where ${\rm A}$  is a vector space, $[\cdot,\cdot]$ is an anticommutative multiplication and $\circ$ is a Jacobi-Jordan multiplication on ${\rm A},$  such that these two bilinear maps satisfy the anti-Leibniz identity:
\begin{longtable}{lcl}
$[x,y\circ z]+y\circ [x,z]+[x,y]\circ z$&$ =$&$0.$
\end{longtable} 
\end{definition}

\noindent
We deduce that $[\cdot,\cdot]$ defines an ${\mathcal A}{\rm BD}$-structure on $({\rm A},\circ).$

\begin{definition}[see \cite{dP,R22}] 
    A generic transposed anti-Poisson-Jacobi-Jordan algebra\footnote{We are using the terminology from \cite{dP}, but it was called an anti-Poisson algebra  in \cite{R22}.} is a triple $({\rm A},\circ,[\cdot,\cdot]),$ where ${\rm A}$  is a vector space, $[\cdot,\cdot]$ is an anticommutative multiplication and $\circ$ is a Jacobi-Jordan multiplication on ${\rm A},$  such that these two bilinear maps satisfy the transposed anti-Leibniz identity:
\begin{longtable}{lcl}
$[x,y]\circ z+[x,y\circ z]+[x\circ z,y]$&$=$&$0.$
\end{longtable} 
\end{definition}
\noindent We deduce that $[\cdot,\cdot]$ defines an ${\mathcal A}{\rm BD}$-structure on $({\rm A},\circ)$ if and only if
\begin{equation*}
[x\circ z,y]\ =\ [x,z]\circ y.  
\end{equation*}
This is equivalent to
\begin{equation}\label{int}
(z\cdot x)\cdot  y \ =\ y\cdot (x\cdot z)\footnote{Algebras with the present identity were considered in \cite{aks24} under the name {\it algebras of first associative type with $\sigma=(13)$}.}.
\end{equation}
Let us also  mention that the variety of algebras defined by the identity \eqref{int} is the intersection of the varieties of flexible and ${\mathcal A}$flexible algebras.
\begin{corollary}
	{ 	
		If identity \eqref{int} holds, then every 
        generic transposed anti-Poisson-Jacobi-Jordan admissible algebra $({\rm A},\cdot)$ is an ${\mathcal A}{\rm BD}$-algebra. 				
	}
\end{corollary}

 

		

\section{Subvarieties  of the variety of ${\mathcal A}{\rm BD}$-algebras}\label{subvar}
\subsection{Anti-Leibniz algebras}
\begin{definition}[see \cite{BCM}]
	{ 
		Let $({\rm A},\cdot)$ be an algebra. $({\rm A},\cdot)$ is said to be a left $\big($resp., right$\big)$ anti-Leibniz algebra if, for all $x,y,z\in {\rm A}:$
		\begin{equation}\label{Eq9}
		x\cdot (y\cdot z)=-(x\cdot y)\cdot  z-y\cdot (x\cdot z) 
		\end{equation} 
		\begin{equation}\label{Eq10}
		\big(\mbox{resp.,}\;\; (y\cdot z)\cdot  x=-(y\cdot x)\cdot  z-y\cdot (z\cdot x) \big).
		\end{equation} 
		$({\rm A},\cdot)$ is said to be a symmetric anti-Leibniz algebra if  \eqref{Eq9} and \eqref{Eq10} hold.
		
	}
\end{definition}
\begin{remark}
	{ 
		Any Jacobi-Jordan algebra is a left $\big($resp., right$\big)$ anti-Leibniz algebra. Conversely, a left $\big($resp., right$\big)$ anti-Leibniz algebra is a Jacobi-Jordan algebra if and only if its product is commutative. 		
	}
\end{remark}

\begin{lemma}\label{lemma1} 
	{ 
		Let $({\rm A},\cdot)$ be an algebra. 
        If $\varphi \in {\mathcal A}\mathfrak{Der}({\rm A}, \cdot),$ then $\varphi \in {\mathcal A}\mathfrak{Der}({\rm A}, \circ).$   				
	}
\end{lemma}
\begin{proposition}
	 Every symmetric anti-Leibniz algebra $({\rm A},\cdot)$ is an ${\mathcal A}{\rm BD}$-algebra. 
	  A left $\big($resp., right$\big)$ anti-Leibniz algebra $({\rm A},\cdot)$ is an ${\mathcal A}{\rm BD}$-algebra if and only if it is a Jacobi-Jordan-admissible algebra.
	 
\end{proposition}
\begin{proof}
	 First, let $({\rm A},\cdot)$ be a symmetric anti-Leibniz algebra. Then the identity \eqref{Eq9}  $\big($resp.,  \eqref{Eq10}$\big)$ means that ${\rm L}_{x}$ $\big($resp., ${\rm R}_{x}$$\big)$ $\in {\mathcal A}\mathfrak{Der}({\rm A},\cdot)$. We deduce from Lemma \ref{lemma1} that ${\rm L}_{x}$ (resp., ${\rm R}_{x})$ $\in {\mathcal A}\mathfrak{Der}({\rm A}^{+})$.
		
	Second,  if $({\rm A},\cdot)$ is an ${\mathcal A}{\rm BD}$-algebra, then it is a Jacobi-Jordan-admissible algebra. 
    Conversely, let $({\rm A},\cdot)$ be a Jacobi-Jordan-admissible algebra. As ${\rm L}_{x}$ $\in {\mathcal A}\mathfrak{Der}({\rm A},\cdot)$, by Lemma \ref{lemma1} we have ${\rm L}_{x}$ $\in {\mathcal A}\mathfrak{Der}({\rm A}^{+})$. By Proposition \ref{pr1},  $({\rm A},\cdot)$ is an ${\mathcal A}{\rm BD}$-algebra.  
 
\end{proof}

\subsection{${\mathcal A}{\rm LR}$-algebras }
\begin{definition}
	{ 
		Let $({\rm A},\cdot)$ be an algebra. For $x,y,z \in {\rm A}$, we define the trilinear maps $[\cdot,\cdot,\cdot]_{{\mathcal A}{\rm L}}:{\rm A}\times {\rm A}\times {\rm A}\rightarrow {\rm A}$ and $[\cdot,\cdot,\cdot]_{{\mathcal A}{\rm R}}:{\rm A}\times {\rm A}\times {\rm A}\rightarrow {\rm A}$ by:
		
		$$[x,y,z]_{{\mathcal A}{\rm L}}:=x\cdot (y\cdot z)+y\cdot (x\cdot z)\;\; \mbox{and}\;\; [x,y,z]_{{\mathcal A}{\rm R}}:=(x\cdot y)\cdot  z+(x\cdot z)\cdot  y. $$ Then 
		\begin{enumerate}
			\item [{\rm (a)}] $({\rm A},\cdot)$ is called an ${\mathcal A}{\rm L}$-algebra $\big($anti-left-commutative$\big)$ if $[x,y,z]_{{\mathcal A}{\rm L}}=0.$ 
			\item [{\rm (b)}] $({\rm A},\cdot)$ is called an ${\mathcal A}{\rm R}$-algebra $\big($anti-right-commutative\footnote{Examples of anti-right-commutative algebras are given by fermionic Novikov algebras \cite{BNH}.}$\big)$ if $[x,y,z]_{{\mathcal A}{\rm R}}=0.$ 
			\item [{\rm (c)}] $({\rm A},\cdot)$ is called an ${\mathcal A}{\rm LR}$-algebra $\big($antibicommutative$\big)$ if it is at the same time an ${\mathcal A}{\rm L}$-algebra and an ${\mathcal A}{\rm R}$-algebra.
		\end{enumerate}
	}
\end{definition}
\begin{remark}
	{  $[x,y,z]_{{\mathcal A}{\rm L}}=[y,x,z]_{{\mathcal A}{\rm L}}$  and  $[x,y,z]_{{\mathcal A}{\rm R}}=[x,z,y]_{{\mathcal A}{\rm R}},$$\;\;$ for all $\;x,y,z\in {\rm A}$.
		
	}
\end{remark}
\begin{example}
	{ 
Any anticommutative antiassociative (= Koszul dual Jacobi-Jordan) algebra is an ${\mathcal A}{\rm LR}$-algebra. 
}
\end{example}
We get the following new  characterization of  ${\mathcal A}{\rm BD}$-algebras by means of the trilinear maps $[\cdot,\cdot,\cdot]_{{\mathcal A}{\rm L}}$ and $[\cdot,\cdot,\cdot]_{{\mathcal A}{\rm R}}$. 
\begin{proposition}
	{ 		
		Let $({\rm A},\cdot)$ be an algebra. Then $({\rm A},\cdot)$ is an ${\mathcal A}{\rm BD}$-algebra if and only if 
		\begin{equation}\label{Def4.}
		[x,y,z]_{{\mathcal A}{\rm L}}+[x,z,y]_{{\mathcal A}{\rm L}}+[x,y,z]_{{\mathcal A}{\rm R}}\ =\ 0,	
		\end{equation} 
		\begin{equation}\label{Def5.}
		[y,z,x]_{{\mathcal A}{\rm R}}+[z,y,x]_{{\mathcal A}{\rm R}}+[z,y,x]_{{\mathcal A}{\rm L}}\ =\ 0.	
		\end{equation} 		
	}
\end{proposition}
\begin{proof}
	{\rm
Using the identity \eqref{Def2.},
$${\mathcal A}(x,y,z)+{\mathcal A}(x,z,y)\ =\ -z\cdot (x\cdot y)-y\cdot (x\cdot z),$$
this is equivalent to $$(x\cdot y)\cdot  z+x\cdot (y\cdot z)+(x\cdot z)\cdot  y+x\cdot (z\cdot y)\ =\ -z\cdot (x\cdot y)-y\cdot (x\cdot z)$$
Then  \eqref{Def4.} holds. Similarly, to get  \eqref{Def5.} we use  \eqref{Def3.}.		
		
}
\end{proof}		
As a direct consequence, we deduce the following corollary,
which generalizes Corollary \ref{dml}.
\begin{corollary}\label{corALR}
	{ 
		Any ${\mathcal A}{\rm LR}$-algebra is an ${\mathcal A}{\rm BD}$-algebra.  			
	}
\end{corollary}
\begin{corollary}
	{ 
		Let $({\rm A},\cdot)$ be an ${\mathcal A}{\rm BD}$-algebra. The following assertions are equivalent:				
	}
\begin{enumerate}
	\item [{\rm (a)}]  $({\rm A},\cdot)$ is an ${\mathcal A}{\rm L}$-algebra$;$ 
	\item[{\rm (b)}]   $({\rm A},\cdot)$ is an ${\mathcal A}{\rm R}$-algebra$;$
	\item[{\rm (c)}]    $({\rm A},\cdot)$ is an ${\mathcal A}{\rm LR}$-algebra$.$
\end{enumerate}
\end{corollary}

\begin{theorem}\label{s4adm}
		Let $({\rm A},\cdot)$ be an ${\mathcal A}{\rm LR}$-algebra. Then
        \begin{enumerate}[I.]
            \item[{\rm (1)}] the derived algebra  ${\rm A}^2$ is associative and commutative.
            \item[{\rm (2)}]    $({\rm A}, [\cdot,\cdot])$ is a metabelian algebra, i.e. $\big[[x,y],[z,t]\big]=0$\footnote{Let us note that each bicommutative algebra under the commutator product gives a metabelian Lie algebra and the derived algebra of each bicommutative algebra is also associative and commutative.}. In particular, 
             $({\rm A}, [\cdot,\cdot])$ is a Lie  algebra if and only if $({\rm A}, \cdot)$ satisfies 
             \begin{center}
             $(x_1,x_2,x_3)+(x_2,x_3,x_1)+(x_3,x_1,x_2)$$ \ = \ 0.$
             \end{center}
            \item[{\rm (3)}] 
$({\rm A},\circ, [\cdot,\cdot])$ is a generic anti-Poisson-Jacobi-Jordan algebra whose anticommutative part $({\rm A},[\cdot,\cdot])$ is metabelian.

        \end{enumerate}
\end{theorem}

\begin{proof}
First, note that the following relations hold in $({\rm A},\cdot):$
\begin{longtable}{llllllllllllllll}
$(x\cdot y)\cdot (z\cdot t)$ &$=$&$ -\  \big(x\cdot (z\cdot t)\big)\cdot y$ &$=$& $\big(z\cdot (x\cdot t)\big)\cdot y$ &$=$&
$-(z\cdot y)\cdot (x\cdot t)$& $=$ \\
&&& $=$ &$x\cdot \big((z\cdot y)\cdot t\big)$ &$=$& $-x\cdot \big((z\cdot t)\cdot y\big)$ &$=$& $(z\cdot t)\cdot (x\cdot y),$
\end{longtable}
\noindent  and 
\begin{longtable}{llllllllllllllll}
$\big((x\cdot y)\cdot (z\cdot t)\big)\cdot (v\cdot w)$ &$=$&$
-\ \big(z\cdot ((x\cdot y)\cdot t)\big)\cdot (v\cdot w)$ &$=$&$
\big(z\cdot (v\cdot w)\big)\cdot \big((x\cdot y)\cdot t\big)$ \\
&$=$&$
-\ (x\cdot y)\cdot \big((z\cdot (v\cdot w))\cdot t\big)$ &$=$&$
(x\cdot y)\cdot \big((z\cdot t)\cdot (v\cdot w)\big)$.
\end{longtable} 
\noindent Thus, the derived algebra ${\rm A}^2$ is associative and commutative.

Second, 
\begin{longtable}{llllllllllll}
$J(x_1,x_2,x_3)$ & $=$&
\multicolumn{5}{l}{$\big[[x_1,x_2],x_3\big]+\big[[x_2,x_3],x_1\big]+\big[[x_3,x_1],x_2\big]$$ \ =$}\\
& $=$ & \multicolumn{5}{l}{$(x_1\cdot x_2)\cdot x_3-(x_2\cdot x_1)\cdot x_3 - x_3\cdot (x_1\cdot x_2)+x_3\cdot (x_2\cdot x_1) +$}\\
&   & \multicolumn{5}{l}{$(x_2\cdot x_3)\cdot x_1-(x_3\cdot x_2)\cdot x_1 - x_1\cdot (x_2\cdot x_3)+x_1\cdot (x_3\cdot x_2) +$}\\
&   & \multicolumn{5}{l}{$(x_3\cdot x_1)\cdot x_2-(x_1\cdot x_3)\cdot x_2 - x_2\cdot (x_3\cdot x_1)+x_2\cdot (x_1\cdot x_3) \ = $}\\
&  $=$ & \multicolumn{5}{l}{$2 \big( (x_1\cdot x_2)\cdot x_3-x_3\cdot (x_1\cdot x_2) + 
 (x_2\cdot x_3)\cdot x_1-$}\\ 
 \multicolumn{6}{r}{$x_1\cdot (x_2\cdot x_3) +
  (x_3\cdot x_1)\cdot x_2-x_2\cdot (x_3\cdot x_1) \big) \ = $}\\
&   & \multicolumn{5}{r}{$2 \big( (x_1,x_2,x_3)  +  (x_2,x_3,x_1)+(x_3,x_1,x_2) \big)$;}\\

\end{longtable}

Further, since $[x,y]$ and $[z,t]$ belong to ${\rm A}^2$ and ${\rm A}^2$ is commutative by {\rm (1)}, we have
$\big[[x,y],[z,t]\big]=[x,y]\cdot[z,t]-[z,t]\cdot[x,y]=0$, i.e. $({\rm A},[\cdot,\cdot])$ is metabelian.
Finally, by Corollary \ref{corALR} and Proposition \ref{pr2}, $[\cdot,\cdot]$ is a skew-symmetric anti-biderivation of ${\rm A}^+$,
i.e. $({\rm A},\circ,[\cdot,\cdot])$ is a generic anti-Poisson-Jacobi-Jordan algebra.

\end{proof}

\begin{theorem}\label{alrid}
Let $({\rm A},\cdot)$ be an ${\mathcal A}{\rm LR}$-algebra. Then $({\rm A},[\cdot,\cdot])$ satisfies the identity
\begin{equation}\label{Fid}
\Big[\big[[x,y],z\big],t\Big]+\Big[\big[[x,t],z\big],y\Big]+\Big[\big[[z,y],x\big],t\Big]+\Big[\big[[z,t],x\big],y\Big]\ =\ 0.
\end{equation}
Moreover, over a field of characteristic different from $2$, every multilinear identity that holds in the commutator algebras of all ${\mathcal A}{\rm LR}$-algebras
is a consequence of the anticommutativity, of the metabelian identity $\big[[x,y],[z,t]\big]=0$ and of \eqref{Fid}. More precisely,
write $[x_1,x_2,\ldots,x_n]=\big[\ldots\big[[x_1,x_2],x_3\big],\ldots,x_n\big]$ for the left-normed commutators. Then
\begin{enumerate}
\item[{\rm (1)}] in degree $3$ there are no identities except those following from the anticommutativity;
\item[{\rm (2)}] in degree $4$ the identities form an $8$-dimensional subspace of the $15$-dimensional multilinear component of degree $4$ of the free anticommutative algebra:
it is spanned by the $3$ metabelian identities $\big[[x_a,x_b],[x_c,x_d]\big]=0$ and by the $5$ linearly independent substitutions of \eqref{Fid}
\[(x,y,z,t)\ \in\ \big\{(x_1,x_2,x_3,x_4),\ (x_1,x_2,x_4,x_3),\ (x_1,x_3,x_2,x_4),\ (x_2,x_1,x_3,x_4),\ (x_2,x_1,x_4,x_3)\big\};\]
\item[{\rm (3)}] in degree $n\geq5$, a multilinear element $\sum_{\sigma\in{\mathbb S}_n}c_\sigma[x_{\sigma(1)},x_{\sigma(2)},\ldots,x_{\sigma(n)}]$ is an identity if and only if
\[\sum_{\sigma\in{\mathbb S}_n,\ k\in\{\sigma(1),\sigma(2)\}}{\rm sgn}(\sigma)\,c_\sigma\ =\ 0\qquad\mbox{for}\ k=1,\ldots,n.\]
In particular, the multilinear component of degree $n$ of the relatively free algebra has dimension $n$.
\end{enumerate}
Hence, over a field of characteristic zero, all identities of the commutator algebras of ${\mathcal A}{\rm LR}$-algebras
follow from the anticommutativity, the metabelian identity and \eqref{Fid}.
\end{theorem}
\begin{proof}
For $a\in{\rm A}$ let ${\rm ad}_a={\rm L}_a-{\rm R}_a$, i.e. ${\rm ad}_a(b)=[a,b]$. The ${\mathcal A}{\rm L}$ and ${\mathcal A}{\rm R}$ identities mean that
${\rm L}_a{\rm L}_b=-{\rm L}_b{\rm L}_a$ and ${\rm R}_a{\rm R}_b=-{\rm R}_b{\rm R}_a$; hence
\[{\rm ad}_a{\rm ad}_b+{\rm ad}_b{\rm ad}_a\ =\ -\,T_{a,b},\qquad T_{a,b}\ =\ {\rm L}_a{\rm R}_b+{\rm R}_b{\rm L}_a+{\rm L}_b{\rm R}_a+{\rm R}_a{\rm L}_b .\]
Since $\big[[[x,y],z],t\big]={\rm ad}_t{\rm ad}_z{\rm ad}_x(y)$, the left hand side of \eqref{Fid} equals
\[{\rm ad}_t\big({\rm ad}_z{\rm ad}_x+{\rm ad}_x{\rm ad}_z\big)(y)+{\rm ad}_y\big({\rm ad}_z{\rm ad}_x+{\rm ad}_x{\rm ad}_z\big)(t)\ =\
-\Big(\big[t,T_{x,z}(y)\big]+\big[y,T_{x,z}(t)\big]\Big),\]
and a direct computation with the ${\mathcal A}{\rm L}$ and ${\mathcal A}{\rm R}$ identities shows that $\big[t,T_{x,z}(y)\big]$ is skew-symmetric in $y$ and $t$.
Hence \eqref{Fid} holds.

Let ${\mathfrak M}$ be the variety of anticommutative algebras satisfying the metabelian identity and \eqref{Fid}; by Theorem \ref{s4adm} and the first part of the proof,
it contains the commutator algebras of all ${\mathcal A}{\rm LR}$-algebras.
In an algebra of ${\mathfrak M}$ the product of two elements of the square vanishes, hence every multilinear element is a linear combination of left-normed commutators.

To show that there are no other identities, we use the following ${\mathcal A}{\rm LR}$-algebras. Let $V$ be a vector space with a basis $v_1,\ldots,v_m$ and let
${\rm G}$ be the Grassmann algebra of $V$, i.e. the associative algebra with unit $1$ generated by $v_1,\ldots,v_m$ subject to the relations $v_iv_j=-v_jv_i$;
the products $v_{i_1}v_{i_2}\cdots v_{i_r}$ with $i_1<i_2<\ldots<i_r$ form a basis of ${\rm G}$. Let ${\rm W}_m=V\oplus{\rm G}$ with the product
$(v+a)\cdot(v'+a')=va'$ for $v,v'\in V$, $a,a'\in{\rm G}$ (the product on the right hand side is taken in ${\rm G}$).
It is an ${\mathcal A}{\rm LR}$-algebra: ${\rm W}_m\cdot{\rm W}_m\subseteq{\rm G}$ and ${\rm G}\cdot{\rm W}_m=0$, so $(x\cdot y)\cdot z=0$, and
$x\cdot(y\cdot z)+y\cdot(x\cdot z)=(v_xv_y+v_yv_x)a_z=0$, where $x=v_x+a_x$, etc. In ${\rm W}_m$ we have
\[[a,b]=0,\quad [v,v']=0,\quad [a,v]=-va,\quad [1,v]=-v\qquad (a,b\in{\rm G},\ v,v'\in V).\]

\begin{enumerate}
    \item[{\rm (1)}] In ${\rm W}_2$ the substitutions $(v_1,v_2,1)$, $(v_1,1,v_2)$ and $(1,v_1,v_2)$ into \begin{center}
$c_1\big[[x_1,x_2],x_3\big]+c_2\big[[x_2,x_3],x_1\big]+c_3\big[[x_3,x_1],x_2\big]$
    \end{center} give
$c_2+c_3=c_1+c_2=c_1+c_3=0$ $\big($coefficients of $v_1v_2\big)$, i.e. $c_1=c_2=c_3=0$.

\item[{\rm (2)}] Every multilinear element of degree $4$ of the free anticommutative algebra is a unique linear combination of the $12$ monomials
$\big[[[x_a,x_b],x_c],x_d\big]$ with $a<b$ and of the $3$ monomials $\big[[x_a,x_b],[x_c,x_d]\big]$. The metabelian identities hold by Theorem \ref{s4adm},
and the five substitutions of \eqref{Fid} listed above are linearly independent.
Besides ${\rm W}_3$ $\big($of dimension $11\big)$ we use the $13$-dimensional algebra ${\rm Q}$ with basis $e_1,e_2,e_3,e_4,e_{13},e_{14},e_{23},e_{24},e_{123},e_{124},e_{134},e_{234},e_{1234}$ and the non-zero products
\[\begin{array}{c}
e_1e_3=e_{13},\quad e_1e_4=e_{14},\quad e_2e_3=e_{23},\quad e_2e_4=e_{24},\\
e_2e_{13}=-e_1e_{23}=e_{123},\quad e_2e_{14}=-e_1e_{24}=e_{124},\\
e_{13}e_4=-e_{14}e_3=e_{134},\quad e_{23}e_4=-e_{24}e_3=e_{234},\\
e_1e_{234}=-e_2e_{134}=e_{13}e_{24}=-e_{14}e_{23}=e_{1234},\\
-e_{23}e_{14}=e_{24}e_{13}=-e_{123}e_4=e_{124}e_3=e_{1234}.
\end{array}\]
Substituting $v_1,v_2,v_3,1$ $\big($in four different orders$\big)$ into a multilinear element of degree $4$ in ${\rm W}_3$ and taking the coefficient of $v_1v_2v_3$, and substituting
$e_1,e_2,e_3,e_4$ $\big($in three different orders$\big)$ in ${\rm Q}$ and taking the coefficient of $e_{1234}$, we obtain seven linear conditions on the coefficients of an identity;
they are linearly independent if the characteristic is different from $2$. Hence the space of identities of degree $4$ has dimension at most $15-7=8$, which proves {\rm (2)}.

\item[{\rm (3)}] Let $n\geq5$ and let ${\rm M}\in{\mathfrak M}$. Substituting $z=w$ and $y=w$ with $w\in{\rm M}^2$ in \eqref{Fid} and using the metabelian identity, we get
\[\big[[[w,y],x],t\big]=-\big[[[w,t],x],y\big],\qquad \big[[[w,x],z],t\big]=-\big[[[w,z],x],t\big].\]
Since the transposition of the last two arguments is the product $(a\, b)(a\, c)(a\, b)$ of the other two transpositions, $\big[[[w,a],b],c\big]$ is skew-symmetric in $a,b,c$ for every $w\in{\rm M}^2$.
Consequently, for $x_1,\ldots,x_n\in{\rm M}$ the commutator $[x_{\sigma(1)},\ldots,x_{\sigma(n)}]$ is skew-symmetric with respect to the positions $3,\ldots,n$, so that
\[[x_{\sigma(1)},x_{\sigma(2)},\ldots,x_{\sigma(n)}]\ =\ {\rm sgn}(\sigma)\,g\big(\sigma(1),\sigma(2)\big)\]
for some $g(a,b)=g(b,a)\in{\rm M}$ (by the anticommutativity). Applying \eqref{Fid} to $x_a,x_b,x_c,x_d$ and multiplying by the remaining elements, we get
\begin{center}$g(a,b)-g(a,d)-g(c,b)+g(c,d)=0$
\end{center} for pairwise distinct $a,b,c,d$.
For distinct $a,b\in\{2,\ldots,n\}$ put 
\begin{center}$\kappa(a,b)=g(a,b)-g(1,a)-g(1,b)$;\end{center} clearly $\kappa(a,b)=\kappa(b,a)$.
For pairwise distinct $a,b,d\in\{2,\ldots,n\}$, it  gives \begin{center}
$g(a,b)-g(a,d)-g(1,b)+g(1,d)=0$, i.e. $\kappa(a,b)=\kappa(a,d)$.
\end{center}
Thus $\kappa$ does not change when one element of the pair is replaced by another element of $\{2,\ldots,n\}$. Any pair $\{a,b\}$ can be turned into any other pair $\{c,d\}$ by at most two such replacements
$\big(\{a,b\}\to\{a,d\}\to\{c,d\}$ if $\{a,b\}\cap\{c,d\}=\varnothing\big)$, so $\kappa(a,b)=\kappa$ is the same for all pairs. Since the characteristic is different from $2$, we can put
\[h_1=-\tfrac12\kappa,\qquad h_a=g(1,a)+\tfrac12\kappa\quad(a\neq1).\]
Then $h_1+h_a=g(1,a)$ and $h_a+h_b=g(1,a)+g(1,b)+\kappa=g(a,b)$ for $a,b\neq1$.

\[\sum_{\sigma\in{\mathbb S}_n}c_\sigma[x_{\sigma(1)},\ldots,x_{\sigma(n)}]\ =\ \sum_{\sigma\in{\mathbb S}_n}{\rm sgn}(\sigma)\,c_\sigma\big(h_{\sigma(1)}+h_{\sigma(2)}\big)\ =\ \sum_{k=1}^n C_k\,h_k,\]
where
\[C_k=\sum_{\sigma\in{\mathbb S}_n,\ k\in\{\sigma(1),\sigma(2)\}}{\rm sgn}(\sigma)\,c_\sigma .\]
If $C_1=\ldots=C_n=0$, the element vanishes for all $x_1,\ldots,x_n$ in every algebra of ${\mathfrak M}$, in particular in the commutator algebra of every ${\mathcal A}{\rm LR}$-algebra.

Assume that the element is an identity of the commutator algebras of all ${\mathcal A}{\rm LR}$-algebras, and fix $k$.
Since ${\rm W}_n$ is an ${\mathcal A}{\rm LR}$-algebra, we take $x_k=1$ and $x_j=v_j$ for $j\neq k$.
To avoid confusion, for $v\in V$ denote by $\bar v$ the same vector regarded as an element of degree $1$ of ${\rm G}$. By the definition of the product, $v\cdot1=\bar v$, $1\cdot v=0$, $a\cdot v=0$ and $v\cdot a=\bar va$ for $a\in{\rm G}$; hence
\[[v,v']=0,\qquad [1,v]=-\bar v,\qquad [a,v]=-\bar v\,a\qquad(v,v'\in V,\ a\in{\rm G}).\]
(i) If $k\notin\{\sigma(1),\sigma(2)\}$, then $f(\sigma)$ starts with $[v_{\sigma(1)},v_{\sigma(2)}]=0$, so $f(\sigma)=0$. Hence $g(a,b)=h_a+h_b=0$ for all distinct $a,b\neq k$.
As there are $n-1\geq3$ indices different from $k$, for pairwise distinct $a,b,c\neq k$ we get $h_a+h_b=h_a+h_c=h_b+h_c=0$, so $2h_a=0$ and $h_a=0$. Thus $h_j=0$ for all $j\neq k$.

(ii) Let $\sigma(1)=k$ and $\big(\sigma(2),\ldots,\sigma(n)\big)=(j_1,\ldots,j_{n-1})$. Applying the above formulas step by step,
\[f(\sigma)=[1,v_{j_1},v_{j_2},\ldots,v_{j_{n-1}}]=(-1)^{n-1}\,\bar v_{j_{n-1}}\cdots\bar v_{j_2}\bar v_{j_1},\]
which is, up to sign, a basis element of ${\rm G}$, because the indices $j_1,\ldots,j_{n-1}$ are pairwise distinct. Hence $h_k=h_k+h_{j_1}=g(k,j_1)=\pm f(\sigma)\neq0$.

Hence  the value of the identity at this substitution is $\sum_jC_jh_j=C_kh_k$. It must be zero, and $h_k\neq0$, so $C_k=0$. This holds for every $k=1,\ldots,n$.

No we prove that every multilinear identity of degree $n$ follows from the metabelian identity and \eqref{Fid}.
Let $P(x_1,\ldots,x_n)$ be a multilinear identity of the commutator algebras of all ${\mathcal A}{\rm LR}$-algebras, and write $P=P_0+P_1$ as explained before the proof of {\rm (1)}, with
$P_0=\sum_\sigma c_\sigma[x_{\sigma(1)},\ldots,x_{\sigma(n)}]$ and $P_1$ that follows from the metabelian condition. 
As noted there, $P_0$ is also an identity of these commutator algebras, so $C_1=\ldots=C_n=0$.
Now consider  the algebra ${\rm M}={\rm F}_{\mathfrak M}(x_1,\ldots,x_n)$, which belongs to ${\mathfrak M}$, and to its free generators $x_1,\ldots,x_n$.
It follows  $P_0(x_1,\ldots,x_n)=\sum_kC_kh_k=0$ in ${\rm F}_{\mathfrak M}(x_1,\ldots,x_n)$, i.e. $P$ is a consequence of the anticommutativity, the metabelian identity and \eqref{Fid}.
 
Every multilinear element of degree $n$ is congruent, modulo the metabelian condition, to an element $\sum_\sigma c_\sigma[x_{\sigma(1)},\ldots,x_{\sigma(n)}]$.
Therefore the multilinear component of degree $n$ of the relatively free algebra is the image of the map $(c_\sigma)_{\sigma\in{\mathbb S}_n}\mapsto\sum_\sigma c_\sigma[x_{\sigma(1)},\ldots,x_{\sigma(n)}]$, whose kernel is the kernel of the linear map $(c_\sigma)\mapsto(C_1,\ldots,C_n)\in{\mathbb F}^n$.
So the dimension equals the rank of the latter map. Taking $c_\sigma=1$ for one $\sigma$ with $\sigma(1)=a$, $\sigma(2)=b$ and $c_\tau=0$ otherwise, we get $\pm(e_a+e_b)$, where $e_1,\ldots,e_n$ is the standard basis of ${\mathbb F}^n$.
Since $e_a=\frac12\big((e_a+e_b)+(e_a+e_c)-(e_b+e_c)\big)$ for distinct $a,b,c$, these vectors span ${\mathbb F}^n$, and the dimension is $n$.

\end{enumerate}
Finally, over a field of characteristic zero every identity is equivalent to a set of multilinear identities of the same degree.\end{proof}

\begin{proposition}
    Let $\mathcal{A}{\rm R}$ be the operad governed by the variety of $\mathcal{A}{\rm R}$-algebras. Then the dual operad $\mathcal{A}{\rm R}^!$ is governed by the variety of $2$-step left nilpotent $\mathcal{A}{\rm R}$-algebras.
  
\end{proposition}

\begin{proof}
    Let us construct multilinear basis elements of degree $3$ for the free $\mathcal{A}{\rm R}$-algebra. Below, we give a presentation of $3$ non-basis elements of degree $3$ as  linear
combinations of the basis elements of degree $3$:

\begin{longtable}{rcl}
$(a\circ c)\circ b \ =\ -\  (a\circ b)\circ c,$& $\quad$&$  
(b\circ c)\circ a \ =\ - \ (b\circ a)\circ c,\quad   \quad \quad 
(c\circ b)\circ a \ =\ - \ (c\circ a)\circ b.$\\
\end{longtable}

Following the approach in \cite{GK94}, we compute the dual operad $\mathcal{A}{\rm R}^!$, where the operad $\mathcal{A}{\rm R}$ governs the variety of $\mathcal{A}{\rm R}$-algebras. Then,
 
\begin{longtable}{rcl}
$[[a \otimes x, b \otimes y], c \otimes z] $&$+$&$ [[b \otimes y, c \otimes z], a \otimes x] \ +\  [[c \otimes z, a \otimes x], b \otimes y]$ \\
&$=$&  $\big((a \circ b) \circ c\big) \otimes \big((x \bullet y) \bullet z\big) - \big(c \circ (a \circ b)\big) \otimes (z \bullet (x \bullet y)\big)$ \\
&&$\quad - \big((b \circ a) \circ c\big) \otimes \big((y \bullet x) \bullet z\big) + \big(c \circ (b \circ a)\big) \otimes \big(z \bullet (y \bullet x)\big)$ \\
&&$\quad + \big((b \circ c) \circ a\big) \otimes \big((y \bullet z) \bullet x\big) - \big(a \circ (b \circ c)\big) \otimes \big(x \bullet (y \bullet z)\big)$ \\
&&$\quad - \big((c \circ b) \circ a\big) \otimes \big((z \bullet y) \bullet x\big) + \big(a \circ (c \circ b)\big) \otimes \big(x \bullet (z \bullet y)\big)$ \\
&&$\quad + \big((c \circ a) \circ b\big) \otimes \big((z \bullet x) \bullet y\big) - \big(b \circ (c \circ a)\big) \otimes \big(y \bullet (z \bullet x)\big)$ \\
&&$\quad - \big((a \circ c) \circ b\big) \otimes \big((x \bullet z) \bullet y\big) + \big(b \circ (a \circ c)\big) \otimes \big(y \bullet (x \bullet z)\big)\ =$\\

\multicolumn{3}{l}{$\mbox{
\big(
by using the above presentation of the $3$ non-basis elements of degree $3$}$}\\ 
\multicolumn{3}{r}{$\mbox{
as a linear combination of basis elements, we have \big)}$}\\

\multicolumn{3}{l}{$= \ \big((a \circ b) \circ c\big) \otimes \big((x \bullet y) \bullet z +(x \bullet z) \bullet y\big) 
-\big((b \circ a) \circ c\big) \otimes \big((y \bullet x) \bullet z+(y \bullet z) \bullet x\big)+
$}\\
\multicolumn{3}{l}{$\quad\big((c \circ a) \circ b\big) \otimes \big((z \bullet y) \bullet x+(z \bullet x) \bullet y\big)+ $}\\
\multicolumn{3}{l}{$\quad \big(c \circ (b \circ a)\big) \otimes \big(z \bullet (y \bullet x)\big) - \big(c \circ (a \circ b)\big) \otimes \big(z \bullet (x \bullet y)\big)  - \big(a \circ (b \circ c)\big) \otimes \big(x \bullet (y \bullet z)\big)+$}\\
\multicolumn{3}{l}{$\quad\big(a \circ (c \circ b)\big) \otimes \big(x \bullet (z \bullet y)\big)-\big(b \circ (c \circ a)\big) \otimes \big(y \bullet (z \bullet x)\big)+\big(b \circ (a \circ c)\big) \otimes \big(y \bullet (x \bullet z)\big)$.}

\end{longtable}
Therefore, the Lie-admissibility condition gives us defining identities for the dual operad $\mathcal{A}{\rm R}^!$, which are equivalent to the following:
\begin{longtable}{rclrcl}
$(x \bullet y) \bullet z +(x \bullet z) \bullet y$&$=$&$
0,$ & 
$x \bullet (y \bullet z) $&$=$&$ 0.$
\end{longtable}
\end{proof}

\begin{proposition}\label{selfdual}
    Let $\mathcal{A}{\rm LR}$ be the operad governed by the variety of $\mathcal{A}{\rm LR}$-algebras. Then the dual operad $\mathcal{A}{\rm LR}^!$ is governed by the variety of $\mathcal{A}{\rm LR}$-algebras.
  
\end{proposition}

\begin{proof}
    Let us construct multilinear basis elements of degree $3$ for the free $\mathcal{A}{\rm LR}$-algebra. Below, we give a presentation of $6$ non-basis elements of degree $3$ as  linear
combinations of the basis elements of degree $3$:

\begin{longtable}{rcl}
$(a\circ c)\circ b \ =\ -\  (a\circ b)\circ c,$& $\quad$&$  
(b\circ c)\circ a \ =\ - \ (b\circ a)\circ c,\quad   \quad \quad 
(c\circ b)\circ a \ =\ - \ (c\circ a)\circ b,$\\

$b \circ (a\circ c)  \ =\ -\  a\circ (b\circ c),$& $\quad$&$  
a \circ (c\circ b)   \ =\ - \ c \circ (a\circ b),\quad   \quad \quad 
b \circ (c\circ a)  \ =\ - \ c \circ (b\circ a).$\\

\end{longtable}

Following the approach in \cite{GK94}, we compute the dual operad $\mathcal{A}{\rm LR}^!$, where the operad $\mathcal{A}{\rm LR}$ governs the variety of $\mathcal{A}{\rm LR}$-algebras. Then,
 
\begin{longtable}{rcl}
$[[a \otimes x, b \otimes y], c \otimes z] $&$+$&$ [[b \otimes y, c \otimes z], a \otimes x] \ +\  [[c \otimes z, a \otimes x], b \otimes y]$ \\
&$=$&  $\big((a \circ b) \circ c\big) \otimes \big((x \bullet y) \bullet z\big) - \big(c \circ (a \circ b)\big) \otimes \big(z \bullet (x \bullet y)\big)$ \\
&&$\quad - \big((b \circ a) \circ c\big) \otimes \big((y \bullet x) \bullet z\big) + \big(c \circ (b \circ a)\big) \otimes \big(z \bullet (y \bullet x)\big)$ \\
&&$\quad + \big((b \circ c) \circ a\big) \otimes \big((y \bullet z) \bullet x\big) - \big(a \circ (b \circ c)\big) \otimes \big(x \bullet (y \bullet z)\big)$ \\
&&$\quad - \big((c \circ b) \circ a\big) \otimes \big((z \bullet y) \bullet x\big) + \big(a \circ (c \circ b)\big) \otimes \big(x \bullet (z \bullet y)\big)$ \\
&&$\quad + \big((c \circ a) \circ b\big) \otimes \big((z \bullet x) \bullet y\big) - \big(b \circ (c \circ a)\big) \otimes \big(y \bullet (z \bullet x)\big)$ \\
&&$\quad - \big((a \circ c) \circ b\big) \otimes \big((x \bullet z) \bullet y\big) + \big(b \circ (a \circ c)\big) \otimes \big(y \bullet (x \bullet z)\big)\ =$\\

\multicolumn{3}{l}{$\mbox{
\big(
by using the above presentation of the $6$ non-basis elements of degree $3$}$}\\ 
\multicolumn{3}{r}{$\mbox{
as a linear combination of basis elements, we have \big)}$}\\

\multicolumn{3}{l}{$= \ \big((a \circ b) \circ c\big) \otimes \big((x \bullet y) \bullet z +(x \bullet z) \bullet y\big)  - \big(a \circ (b \circ c)\big) \otimes \big(x \bullet (y \bullet z)+y \bullet (x \bullet z)\big)+$}\\
\multicolumn{3}{l}{$\quad  \ \big(c \circ (b \circ a)\big) \otimes \big(z \bullet (y \bullet x)+y \bullet (z \bullet x)\big) -\big((b \circ a) \circ c\big) \otimes \big((y \bullet x) \bullet z+(y \bullet z) \bullet x\big)+$}\\
\multicolumn{3}{l}{$\quad \ 
\big((c \circ a) \circ b\big) \otimes \big((z \bullet y) \bullet x+(z \bullet x) \bullet y\big) - \big(c \circ (a \circ b)\big) \otimes \big(z \bullet (x \bullet y)+x \bullet (z \bullet y)\big). $}\\

\end{longtable}
Therefore, the Lie-admissibility condition gives us defining identities for the dual operad $\mathcal{A}{\rm LR}^!$, which are equivalent to the following:
\begin{longtable}{rclrcl}
$(x \bullet y) \bullet z +(x \bullet z) \bullet y$&$=$&$
0,$ & 
$x \bullet (y \bullet z)  +y \bullet (x \bullet z)$&$=$&$ 0.$
\end{longtable}
\end{proof}

One of the cornerstones of the theory of vertex algebras is the statement known as the Dong
Lemma for quantum fields, which goes back to the paper \cite{Li}.
The classical Dong Lemma for distributions over a Lie algebra lies in the foundation
of vertex algebras theory. Kolesnikov and Sartayev found a necessary and sufficient condition for a variety
of nonassociative algebras with binary operations to satisfy the analogue of the Dong Lemma \cite{11}.
As  was mentioned in \cite{DS}, the operad ${\rm BiCom}$ governed by the variety of bicommutative $\big($or ${\rm LR}$$\big)$ algebras satisfies the Dong property in the sense of \cite{11}.
The following corollary states that the same holds for the  operad $\mathcal{A}{\rm LR}.$

\begin{corollary}
    The operad $\mathcal{A}{\rm LR}$ satisfies the Dong property in the sense of  {\rm \cite{11}}.
\end{corollary}

\begin{proof}
    The operad $\mathcal{A}{\rm LR}$ is self-dual and the monomials $(ab)c,$ $(ba)c,$ $c(ab)$ and $c(ba)$ are linearly independent in $\mathcal{A}{\rm LR}^! \langle X\rangle.$ 
    By the criterion given in \cite{11}, the operad $\mathcal{A}{\rm LR}$ satisfies the Dong
property.
\end{proof}

\subsection{$\mathcal{A}$Flexible ${\mathcal A}{\rm BD}$-algebras }\label{sec:aflex}

\begin{proposition}[see \cite{MH}]
	{ 		
Let $({\rm A},\cdot)$ be an algebra. Then $({\rm A},\cdot)$ is flexible if and only if  
\begin{equation}\label{Leibniz3}
 [x,y\circ z]\ =\ [x,y]\circ z+y\circ [x,z].
\end{equation} 
		
	}
\end{proposition}
\begin{proof}
	{ 	
Note that the identity (\ref{Leibniz3})	is equivalent to:
\begin{equation}\label{Leibniz4}
(x , y , z ) + (z , y , x ) + (y , z , x ) + (x , z , y ) \ =\  (y , x , z ) + (z , x , y ).
\end{equation} 
 If $({\rm A},\cdot)$ is flexible, then the identity \eqref{Leibniz4} holds, and hence the identity \eqref{Leibniz3} holds as well. Conversely, the identity (\ref{Leibniz4}) with $z = x$ clearly gives  the flexibility of $({\rm A},\cdot)$.	
	}
\end{proof} 
\begin{corollary}
	{ 
Let $({\rm A},\cdot)$ be a flexible Jacobi-Jordan-admissible algebra. Then $({\rm A},\cdot)$ is an ${\mathcal A}{\rm BD}$-algebra if and only if $$[x,y\circ z]=0 \  \big(\ i.e., \  ({\rm A}^{+})^{2}\subseteq {\rm Ann}({\rm A}^{-})\ \big).$$

}
\end{corollary} 

\begin{proposition}\label{pr3}
	{ 		
		Let $({\rm A},\cdot)$ be an  algebra. 
        Then $({\rm A},\cdot)$ is $\mathcal{A}$flexible if and only if  
		\begin{equation*}\label{Leibniz5}
		[x\circ y, z]+[x,y\circ z]= [y,x]\circ z+[z,y]\circ x.
		\end{equation*} 		
	}
\end{proposition}
\begin{proof}The proof follows from the consideration of the following equivalent identities.
\begin{longtable}{rcl}
	$(x\cdot y)\cdot z+x\cdot (y\cdot z)$&$=$&$(z\cdot y)\cdot x+z\cdot (y\cdot x)$\\
	$\big([x, y]+x\circ y\big)\cdot z+x\cdot \big([y, z]+y\circ z\big)$&$=$&
    $\big([z, y]+z\circ y\big)\cdot x+z\cdot \big([y, x]+y\circ x\big)$\\
	$[x, y]\circ z+[x\circ y, z]+x\circ [y, z]+[x, y \circ z]$&$=$&$[z, y]\circ x+[z\circ y, x]+z\circ [y, x]+[z, y \circ x]$\\
$[x\circ y, z]+[x, y \circ z]$&$=$&$[y, x]\circ z+   x\circ [z, y].$
\end{longtable}

\end{proof}

\begin{theorem}\label{aflchar}
Let $({\rm A},\cdot)$ be an algebra.
\begin{enumerate}
\item[{\rm (1)}] If $({\rm A},\cdot)$ is an $\mathcal{A}$flexible ${\mathcal A}{\rm BD}$-algebra, then
\begin{equation}\label{aflabd}
x\cdot (y\cdot z) \ =\  (z\cdot y)\cdot  x\footnote{The present identity also appeared in the  depolarization of mixed-Poisson algebras \cite{dP}.};
\end{equation}
in particular, $({\rm A},\cdot)$  is flexible.
\item[{\rm (2)}] Over a field of characteristic different from $2$ and $3$, $({\rm A},\cdot)$ is an
${\mathcal A}$flexible ${\mathcal A}{\rm BD}$-algebra if and only if it satisfies \eqref{aflabd} and $(x\cdot x)\cdot x=0$.
\end{enumerate}
\end{theorem}
\begin{proof}
{\rm (1)} For all $x,y,z\in {\rm A}$, we have the following identity, which is valid in any algebra:
\begin{longtable}{lcr}
$x\circ (y\circ  z)+y\circ (z\circ  x)+z\circ (x\circ  y)
$&$=$&${\mathcal A}(x,y,z)+{\mathcal A}(y,z,x)+{\mathcal A}(z,x,y)+$\\
&&${\mathcal A}(y, x, z)+{\mathcal A}(x, z, y)  +{\mathcal A}(z, y, x).$\\
\end{longtable}
Since $({\rm A},\cdot)$ is an $\mathcal{A}$flexible ${\mathcal A}{\rm BD}$-algebra, we deduce
\begin{longtable}{lcl}
${\mathcal A}(x,y,z)+{\mathcal A}(y,z,x)+{\mathcal A}(z,x,y)$&$=$&$0.$
\end{longtable}
Then, if we replace ${\mathcal A}(x,y,z)$ by
$-{\mathcal A}(x,z,y)-z\cdot (x\cdot y)-y\cdot (x\cdot z)$
$\big($see \eqref{Def2.}$\big)$,
 we have the following:
 \begin{center}
     ${\mathcal A}(z,x,y)-z\cdot (x\cdot y)-y\cdot (x\cdot z)\ = \ 0$,
 \end{center}therefore $x\cdot (y\cdot z) = (z\cdot y)\cdot  x$.
Taking $z=x$ in the last identity, we obtain the flexible identity.

{\rm (2)} If $({\rm A},\cdot)$ is an ${\mathcal A}$flexible ${\mathcal A}{\rm BD}$-algebra, the first identity holds by {\rm (1)} and the second one follows from Proposition \ref{3nil}.
Conversely, assume both identities. Then 
\begin{center}${\mathcal A}(x,y,z)=(x\cdot y)\cdot z+(z\cdot y)\cdot x$\end{center} 
is symmetric in $x,z$, so $({\rm A},\cdot)$ is ${\mathcal A}$flexible;
moreover $x\cdot(x\cdot x)=(x\cdot x)\cdot x=0$, so $({\rm A},\cdot)$ is Jacobi-Jordan-admissible by Proposition \ref{JJadm}. Using $z\cdot (x\cdot y)=(y\cdot x)\cdot z$ and $y\cdot(x\cdot z)=(z\cdot x)\cdot y$, identity \eqref{Def2.} becomes
\[\sum_{\sigma\in{\mathbb S}_3}\big(x_{\sigma(1)}\cdot x_{\sigma(2)}\big)\cdot x_{\sigma(3)}\ =\ 0, \qquad
\text{where}\qquad x_1=x,\ x_2=y,\ x_3=z,\]
which is the full linearization of $(x\cdot x)\cdot x=0$; the same computation applies to \eqref{Def3.}. Hence $({\rm A},\cdot)$ is an ${\mathcal A}{\rm BD}$-algebra.
\end{proof}

\begin{corollary}\label{antiass}
    Let $(\rm A, \cdot)$ be an antiassociative algebra.
    Then $(\rm A, \cdot)$ is an ${\mathcal A}{\rm BD}$-algebra if and only if it satisfies \eqref{aflabd}.
\end{corollary}

		


\begin{proposition}\label{cflex}
	{ 		
		Let $({\rm A},\cdot)$ be an algebra. Then $({\rm A},\cdot)$ is an $\mathcal{A}$flexible ${\mathcal A}{\rm BD}$-algebra if and only if the following assertions are satisfied:
\begin{enumerate}
\item [{\rm (a)}] ${\rm A}^{+}$ is a Jacobi-Jordan algebra$;$ 
\item [{\rm (b)}] $({\rm A}^{+})^{2}\subseteq {\rm Ann}({\rm A}^{-});$
\item [{\rm (c)}] $({\rm A}^{-})^2\subseteq {\rm Ann}({\rm A}^{+}).$
\end{enumerate} 		
	}
\end{proposition}
\begin{proof}
	{\rm
Let $({\rm A},\cdot)$ be an $\mathcal{A}$flexible ${\mathcal A}{\rm BD}$-algebra.

\begin{enumerate}
    \item[{\rm (a)}] 
 By Proposition \ref{pr1}, every ${\mathcal A}{\rm BD}$-algebra is a Jacobi-Jordan-admissible algebra.  
  
\item[{\rm (b)}]  Using identities \eqref{Leibniz3} and \eqref{Eq2}, we get: $[x,y\circ z]=0.$ Then, $({\rm A}^{+})^{2}\subseteq {\rm Ann}({\rm A}^{-})$. 

 \item[{\rm (c)}] Let $x,y,z\in {\rm A}$. We have 
\begin{longtable}{lcl}
$[x,z]\circ y$&$=$&$(x\cdot z)\circ y-(z\cdot x)\circ y$\\
&$=$&$(x\cdot z)\circ y-(z\cdot x)\circ y+(z\cdot y)\circ x-(z\cdot y)\circ x$\\
&$=$&$ (x\cdot z)\circ y+(z\cdot y)\circ x+z\cdot (y\circ x)  \;\;\;\;\;\big(\;\mbox{since}\; {\rm L}_{z}\in {\mathcal A}\mathfrak{Der}({\rm A}^{+})\;\big)$.
\end{longtable}
Moreover, we have $$z\cdot (y\circ x)=z\circ(x\cdot y)+z\cdot (y\cdot x)-(x\cdot y)\cdot  z.$$
It follows that
\begin{longtable}{lclcl}
$[x,z]\circ y$&$
=$&$(x\cdot z)\circ y+(z\cdot y)\circ x+z\circ(x\cdot y)+z\cdot (y\cdot x)-(x\cdot y)\cdot  z$\\
&$=$&$-x\cdot (z\circ y)+(z\cdot y)\circ x+z\cdot (y\cdot x)-(x\cdot y)\cdot  z$\\
&$=$&$-x\cdot (z\cdot y)-x\cdot (y\cdot z)+(z\cdot y)\cdot  x+x\cdot (z\cdot y)+z\cdot (y\cdot x)-(x\cdot y)\cdot  z$\\
\multicolumn{3}{r}{$= \ {\mathcal A}(z,y,x)-{\mathcal A}(x,y,z)$}&$=$&$0.$\\     
\end{longtable}
Then, $({\rm A}^{-})^2 \subseteq {\rm Ann}({\rm A}^{+})$. 
\end{enumerate}

Conversely, assume that {\rm (a)}, {\rm (b)} and {\rm (c)} hold. Both sides of \eqref{Eq2} vanish by {\rm (b)} and {\rm (c)}, so $({\rm A},\cdot)$ is an
${\mathcal A}{\rm BD}$-algebra by {\rm (a)} and Proposition \ref{pr2}. Moreover, by {\rm (b)} and {\rm (c)},
\[4\,{\mathcal A}(x,y,z)\ =\ \big[[x,y],z\big]+\big[x,[y,z]\big]+(x\circ y)\circ z+x\circ(y\circ z),\]
and the right hand side is invariant under $x\leftrightarrow z$ (by the commutativity of $\circ$ and the anticommutativity of $[\cdot,\cdot]$).
Hence $({\rm A},\cdot)$ is ${\mathcal A}$flexible.

}
\end{proof}

The following result is a direct consequence of the above proposition.

\begin{corollary}\label{thflex}
Let $({\rm A}, \circ_{\rm A})$ be a Jacobi-Jordan  algebra, 
$\rm{H}$  a trivial $\rm A$-module and $\beta: {\rm A}\times {\rm A} \rightarrow \rm{H}$  a $2$-cocycle of ${\rm A}$ with values in $\rm{H}$.  Let us consider the central extension 
$\big(\mathcal{A}:= {\rm A} \oplus \rm{H}, \circ \big)$ of $({\rm A},\circ_{\rm A})$ by means of $\beta$. 
The multiplication  $\circ$ on $\mathcal{A}$ is defined by 

$$(x+a) \circ (y+b):= x \circ_{\rm A}y  + \beta(x,y), \quad \forall\, (x,a), (y,b) \in {\rm{A}} \times 
\rm{H}.$$
\noindent Let $[\cdot,\cdot]: \mathcal{A} \times \mathcal{A} \rightarrow \rm{H}$ be an anti-symmetric bilinear map,  and let us consider the product $"\cdot"$ on the vector space $\mathcal{A}$ defined by 
\begin{center}
    $x\cdot y= \frac{1}{2}\big  ([x,y] + x\circ y\big), \quad \quad \forall \, x, y\in \mathcal{A}.$
\end{center}
\noindent
Then $(\mathcal{A},\cdot)$ is an $\mathcal{A}$flexible $\mathcal{A} {\rm BD}$-algebra if and only if 
\begin{equation}\label{eqc}
	[ x\circ_{\rm A}y  + \beta(x,y), z+c] \ =\ 0, \quad \forall \, x, y, z \in {\rm A},\, \forall\, c \in \rm{H}.
\end{equation}
\end{corollary}

\begin{definition}
If $\beta$ and $\circ$ satisfy the condition \eqref{eqc}, the above ${\mathcal A}{\rm BD}$-algebra is called the ${\mathcal A}{\rm BD}$-extension of the Jacobi-Jordan   algebra $({\rm A}, \circ_{\rm A})$.
\end{definition}

\begin{theorem}
\label{flex}
Let $({\rm A}, \cdot)$ be an algebra. 
If $(\rm{A}, \cdot)$ is an $\mathcal{A}$flexible $\mathcal{A}{\rm BD}$-algebra, then $(\rm{A},\cdot)$ is an $\mathcal{A}{\rm BD}$-extension of the quotient Jacobi-Jordan algebra $\rm{A}^+/({\rm{A}}^-)^2$.
\end{theorem}

\begin{proof}
     Let $(\rm{A},\cdot)$ be an $\mathcal{A}$flexible $\mathcal{A}{\rm BD}$-algebra. By Proposition \ref{cflex}, $(\rm{A}^-)^2$ is contained in  ${\rm Ann}(\rm{A}^+)$, so $(\rm{A}^-)^2 $ is a two-sided ideal of $(\rm{A},\cdot)$. Let $V$ be a vector subspace of $\rm{A}$ such that 
     ${\rm A}= V\oplus ({\rm A}^-)^2$. For $x,y\in\rm{A}$, we write $x\circ y= \alpha(x,y)+\beta(x,y),$  where $\alpha(x,y) \in V$ and $\beta(x,y) \in (\rm{A}^-)^2.$ Therefore, $\alpha: \rm{A}\times \rm{A} \rightarrow V$ and $\beta: \rm{A}\times \rm{A} \rightarrow (\rm{A}^-)^2$ are symmetric bilinear maps. Let $v_1,v_2,v_3 \in V$. We have 
\begin{longtable}{lclcl}
$v_1\circ (v_2 \circ v_3) $&$ = $&$ v_1 \circ \big(\alpha(v_2,v_3)+\beta(v_2,v_3)\big) $&$ =$\\
\multicolumn{3}{r}{ $v_1 \circ \alpha(v_2,v_3)$}&$=$&$\alpha\big(v_1,\alpha(v_2,v_3)\big)+\beta\big(v_1,\alpha(v_2,v_3)\big)$.
\end{longtable}
\noindent Since
\begin{center}
    $\oint v_1 \circ (v_2 \circ v_3)\ = \ \oint \alpha\big(v_1,\alpha(v_2,v_3)\big) + \oint \beta\big(v_1,\alpha(v_2,v_3)\big)$,
\end{center}
then
\begin{center}
    $\oint v_1\circ (v_2\circ v_3)\ =\ 0$ if and only if 
    $\oint \alpha\big(v_1,\alpha(v_2,v_3)\big) \ =\ 0\, $ and 
$\,\oint \beta\big(v_1,\alpha(v_2,v_3)\big) \ =\ 0.$
\end{center}
Thus $(V, \circ _V:=\alpha|_{V \times V})$ is a Jacobi-Jordan  algebra and 
$\rm{A}^+$ is its central extension  by means of the $2$-cocycle 
$\beta: V \times V \to (\rm{A}^-)^2$, where  $(\rm{A}^-)^2$ is considered as a trivial $(V, \circ_V)$-module.

From now on, the product $[\cdot,\cdot]$ of $\rm{A}^-$  will be denoted  by $\varphi$. Then for all $v,w,t \in V$ and for all $a,b,c \in (\rm{A}^-)^2$, we have
\begin{center}
    $2 (v+a) \cdot (w+b)\ =\ (v+a) \circ (w+b)+\varphi(v+a,w+b)\ =\ v \circ_V w +\beta(v,w)+\varphi(v+a,w+b),$
    \end{center}
\noindent and, by a direct computation, we get 

\begin{longtable}{lclcl}
 &  &$\big((v+a) \cdot (t+c)\big) \circ (w+b) + (v+a) \circ \big((w+b) \cdot (t+c)\big)+\big((v+a) \circ (w+b) \big) \cdot (t+c)$ & $=$\\
&&$\big(v \circ_V t +\beta(v,t)+\varphi (v+a, t+c)\big) \circ (w+b) 
+(v+a) \circ \big( w \circ_V t +\beta(w,t) +\varphi(w+b, t+c)\big)$ & $+$\\
\multicolumn{3}{r}{$\big( v \circ_V w + \beta(v,w)\big) \cdot (t+c) $} &$=$\\
\multicolumn{3}{r}{$ (v\circ_V t) \circ_V w +\beta(v\circ_V t, w) + v \circ (w\circ_V t) + \beta(v, w\circ_V t) 
+ (v\circ_V w)\circ_V t + \beta(v \circ_V w, t)$}&$ +$\\
\multicolumn{3}{r}{$\varphi\big((v \circ_V w)+\beta(v,w), t+c\big)$}&$=$\\
\multicolumn{3}{r}{$\varphi\big((v+a) \circ (w+b), t+c\big)$}
\end{longtable}

Thus   $(\rm{A}, \cdot)$ is an $\mathcal{A}{\rm BD}$-algebra if and only if  
$\varphi\big((v+a) \circ (w+b), t+c\big)=0$. We
conclude that   $(\rm{A},\cdot)$ is an $\mathcal{A}{\rm BD}$-extension of the 
Jacobi-Jordan algebra $(V, \circ_V)$ by means of $(\beta, [\cdot, \cdot])$; note that $(V,\circ_V)$ is isomorphic to the quotient Jacobi-Jordan algebra $\rm{A}^+/({\rm{A}}^-)^2$. 
\end{proof}

\vskip 0,1 cm

Combining Corollary \ref{thflex} and  Theorem \ref{flex}, we obtain the following characterization of $\mathcal{A}{\rm BD}$-algebras.

\begin{corollary} \label{fabd}
Let $(\rm{A}, \cdot)$ be an algebra. Then $(\rm{A},\cdot)$ is an $\mathcal{A}$flexible $\mathcal{A}{\rm BD}$-algebra if and only if 
$(\rm{A},\cdot)$  is an $\mathcal{A}{\rm BD}$-extension of a Jacobi-Jordan algebra $(V,\circ_V).$
\end{corollary}
 
 Antiassociative algebras are obviously $\mathcal{A}$flexible, so the question of characterizing antiassociative $\mathcal{A}{\rm BD}$-algebras naturally arises. We propose to answer this question as follows.

\begin{theorem}\label{assoabd}
	Let $(\mathcal{A}, \cdot )$ be an algebra. Then $(\mathcal{A}, \cdot )$ is an antiassociative $\mathcal{A}{\rm BD}$-algebra if and only if $(\mathcal {A},\cdot)$  is an $\mathcal{A}{\rm BD}$-extension of a  $2$-step nilpotent Jacobi-Jordan algebra $(\rm{A}, \circ_A)$ by means of $(\beta, [\cdot, \cdot])$ satisfying \eqref{eqc} such that the anticommutative product $[\cdot,\cdot] : (\mathcal{A}= \rm{A}\oplus \rm{H}) \times (\mathcal{A}= \rm{A} \oplus \rm{H}) \, \rightarrow\,  \rm{H}$   satisfies	
	
	\begin{equation}
	\label{condasso1}
	\mathcal{A}_{[\cdot, \cdot]}(x,y,z)\ =\ \beta(y,x\circ z),\quad \forall x, y, z \in \rm{A},
	\end{equation}
	and
	\begin{equation}\label{condasso2}
	\mathcal{A}_{[\cdot, \cdot]}(X,Y,Z)\ =\  0,\quad \mbox{for all}\quad   X, Y, Z \in \mathcal{A},\quad   \mbox{such that}\,\,\{X,Y,Z\}\cap \rm{H}\not= \emptyset,
	\end{equation}	
	where $\mathcal{A}_{[\cdot, \cdot]}$ is the antiassociator of the product $[\cdot, \cdot].$
\end{theorem}
\begin{proof}
 Let $(\mathcal{A},\cdot)$  be an antiassociative $\mathcal{A}{\rm BD}$-algebra. Then, by Theorem \ref{flex}, $(\mathcal{A},\cdot)$ is an $\mathcal{A}{\rm BD}$-extension of a Jacobi-Jordan algebra  $(\rm{A},\circ_A)$ by means of $(\beta,[\cdot, \cdot])$ where $\beta$ is a $2$-cocycle of $\rm{A}$ with values in a trivial $\rm{A}$-module $\rm{H}$ and 
 $[\cdot, \cdot] : (\mathcal{A}= \rm{A} \oplus \rm{H}) \times (\mathcal{A}= \rm{A} \oplus \rm{H}) \, \rightarrow\,  \rm{H}$ is an anticommutative product satisfying condition ($\ref{eqc}$). Let $X= x+a, Y=y+b, Z= z+c$ be elements of $\mathcal{A}$, where $x, y, z \in \rm{A}$ and $a, b, c \in \rm{H}.$ We have  
 \begin{longtable}{lclcl}
 $(X\cdot Y) \cdot Z $&$ =$&$  \frac{1}{4}\Big(\big[[X,Y],Z\big] + (X\circ Y)\circ Z\Big) $&$ = $&$ \frac{1}{4}\Big( \big[[X,Y],Z\big]  + (x\circ y)\circ z + \beta(x\circ y,z) \Big),$\\  
 $X \cdot (Y \cdot Z) $&$=$&$ \frac{1}{4}\Big(\big[X,[Y,Z]\big] + X\circ(Y\circ Z)\Big) $&$=$&$ \frac{1}{4}\Big(\big[X,[Y,Z]\big] + x\circ (y\circ z) + \beta(x,y\circ z)\Big).$
 \end{longtable}
 \noindent By the antiassociativity of $(\mathcal{A}, \cdot)$, we get 
 \begin{center}
     $(x\circ y)\circ z\ =\  -x\circ (y\circ z),$ and $\big[[X,Y],Z\big]  +  \beta(x\circ y,z)\ =\  -\big[X,[Y,Z]\big] - \beta(x,y\circ z).$ 
 \end{center}Consequently, $y\circ(z\circ x)= 0\,$ and $\mathcal{A}_{[\cdot, \cdot]}(X,Y,Z)=\beta(y,x\circ z).$ Therefore $(\rm{A},\circ_A)$ is  a $2$-step nilpotent Jacobi-Jordan algebra and conditions \eqref{condasso1} and  \eqref{condasso2} are satisfied.
 
 Conversely, let us assume that $(\mathcal{A},\cdot)$  is an  $\mathcal{A}{\rm BD}$-extension of a $2$-step nilpotent Jacobi-Jordan algebra $(\rm{A},\circ_A)$ by means of $(\beta,[\cdot, \cdot])$ satisfying conditions \eqref{eqc}, \eqref{condasso1} and   \eqref{condasso2}. Then, by Corollary \ref{fabd}, $(\mathcal{A},\cdot)$ is an $\mathcal{A}$flexible $\mathcal{A}{\rm BD}$-algebra. Besides, by using conditions (\ref{condasso1}) and   (\ref{condasso2}), a simple calculation shows that $(\mathcal{A},\cdot)$ is an antiassociative algebra.  \end{proof}
 
Let us consider an antiassociative $\mathcal{A}{\rm BD}$-algebra $(\mathcal{A},\cdot)$. Then, by Theorem \ref{assoabd}, $(\mathcal{A},\cdot)$  is an $\mathcal{A}{\rm BD}$-extension of a $2$-step nilpotent Jacobi-Jordan algebra $(\rm{A}, \circ_A)$ by means of $(\beta, [\cdot, \cdot])$ satisfying ($\ref{eqc}$) such that the anticommutative product \begin{center}
    $[\cdot,\cdot] : (\mathcal{A}= \rm{A}\oplus \rm{H}) \times (\mathcal{A}= \rm{A} \oplus \rm{H}) \, \rightarrow\,  \rm{H}$  
\end{center} satisfies the equations (\ref{condasso1}) and (\ref{condasso2}). 
It is known that each antiassociative algebra  $(\mathcal{A},\cdot)$ is $3$-step nilpotent, 
hence,   $\mathcal{A}^+$ is $3$-step nilpotent too. 
In the following corollary, we will see under which necessary and sufficient conditions $\mathcal{A}^+$ is also a $2$-step nilpotent Jacobi-Jordan algebra.

\begin{corollary}\label{p1}
  The following assertions are equivalent:
		\begin{enumerate}
			\item[{\rm (1)}]  $\mathcal{A}^+$ is a $2$-step nilpotent Jacobi-Jordan algebra$;$
			\item[{\rm (2)}]  $\beta(x\circ_A y,z)= 0,\quad \forall x, y, z \in \rm{A};$
			\item[{\rm (3)}]  $\mathcal{A}^-$ is antiassociative.
	\end{enumerate}
\end{corollary}
\begin{proof}
  Let $X= x+a, Y=y+b, Z= z+c, T= t+d$ be elements of $\mathcal{A}$, where $x, y, z, t \in \rm{A}$ and $a, b, c, d\in \rm{H}$.  Since
  \begin{center}
      $X\circ(Y\circ Z)=x\circ_A(y\circ_A z)+\beta(x,y\circ_A z)$,
  \end{center} $\mathcal{A}^+$ is a $2$-step nilpotent Jacobi-Jordan algebra if and only if $\beta(x,y\circ_A z)=0$ for all $x,y,z \in \rm{A}$.
  
	Now, the fact that \begin{longtable}{lcl}
	    $\mathcal{A}(X,Y,Z)$&$=$&$\frac{1}{4}\Big( \big[[X,Y],Z\big]+(X\circ Y)\circ Z +\big[X,[Y,Z]\big]+X \circ (Y \circ Z)\Big)$
	\end{longtable}\noindent implies that $\mathcal{A}^+$ is  $2$-step nilpotent  if and only if $[\cdot,\cdot]$ is antiassociative. 
 
\end{proof}
\section{Post-Jacobi-Jordan structures on Jacobi-Jordan algebras }	

\begin{definition}
	{ A post-Jacobi-Jordan structure (or PJJ-structure) on a Jacobi-Jordan algebra $({\rm A},\circ)$ is a bilinear product $x\star y$ satisfying the identities:
\begin{equation}\label{Eq1}	     				
x\star y \  =\  -y\star x,
\end{equation}
\begin{equation}\label{Eq1.}		     		
	x\star (y\circ z) \ =\ - \big( (x\star y)\circ z+y\circ(x\star z) \big),
\end{equation}
\begin{equation}\label{Eq1..}		     			
(x\circ y)\star z \ =\ - \big(x\star (y\star z)+y\star (x\star z) \big).
\end{equation}	
}

\end{definition}
 
\begin{proposition}\label{52}
{ 
For Jacobi-Jordan algebras with zero multiplication, 
PJJ-structures correspond to anticommutative antiassociative algebras\footnote{The algebraic classification of anticommutative antiassociative algebras up to dimension 9 is given in \cite{CKLS}. 
Hence, we have the classification (up to isomorphism) of all non-trivial  post-Jacobi-Jordan structures on trivial $n$-dimensional algebras for $n<9.$}.
}
\end{proposition}
\begin{proof}
	{\rm 
 Using identities \eqref{Eq1}, \eqref{Eq1.} and $x\circ y=0$, 
 we have $$(y\star z)\star x\ =\ -x\star (y\star z)\ =\ y\star (x\star z) \ =\ -y\star (z\star x).$$}
\end{proof}

\begin{proposition}
	\label{p2}
	Let $({\rm A}, \cdot)$ be an $\mathcal{A}$flexible $\mathcal{A}{\rm BD}$-algebra. Then the following properties are equivalent.
	\begin{enumerate}
		\item[{\rm (1)}]  ${\rm A}^{-}$ is an antiassociative algebra$;$
		\item[{\rm (2)}]  The bilinear map $ [\cdot, \cdot]$ is a PJJ-structure on the Jacobi-Jordan algebra ${\rm A}^+$.
	\end{enumerate} 
\end{proposition}
\begin{proof}
	 
Assume first that ${\rm A}^{-}$ is antiassociative. It is easy to see that \eqref{Eq1} and \eqref{Eq1.} are satisfied. Now, for any $x,y,z \in \rm{A}$,  on the one hand, since $\rm{A}$ is $\mathcal{A}$flexible, we have $[x\circ y,z]=0$; on the other hand, since $[\cdot, \cdot]$ is antiassociative, we have 
$[x, [y, z]]+ [y, [x, z]]= 0.$ Therefore, $[\cdot, \cdot]$ is a PJJ-structure on the Jacobi-Jordan algebra $\rm{A}^+$. 

Assume now that $[\cdot, \cdot]$ is a PJJ-structure on a Jacobi-Jordan algebra $\rm{A}^+$. Then for any $x,y,z \in \rm{A}$ we have 
\begin{center} 
$\mathcal{A}_{[\cdot, \cdot]}(x,y,z)\ =\ [[x, y],z]+ [x, [y, z]]\ =\ 
[[x, y],z]-[x, [z, y]]\  \overset{\eqref{Eq1..}}{=}\ [[x, y],z]+[z, [x, y]]\ =\ 0.$
\end{center}
\noindent Thus, $ \rm{A}^{-} $ is an antiassociative algebra. We conclude that the first and second assertions    are equivalent.
 
\end{proof}

The following corollary follows from Corollary \ref{p1} and Proposition \ref{p2}.
\begin{corollary}
	Let $(\rm{A},\cdot)$ be an antiassociative $\mathcal{A}{\rm BD}$-algebra. Then the following assertions are equivalent:
	\begin{enumerate}
		\item[{\rm (1)}]  $\rm{A}^{-}$ is an antiassociative algebra$;$
		\item[{\rm (2)}] $\rm{A}^+$ is a $2$-step nilpotent Jacobi-Jordan algebra$;$
		\item[{\rm (3)}]  $\beta(x\circ y,z)= 0,\quad \forall x, y, z \in \rm{A};$
		\item[{\rm (4)}]  The bilinear map $ [\cdot, \cdot]$ is a PJJ-structure on the Jacobi-Jordan algebra $\rm{A}^+$.
	\end{enumerate}
\end{corollary}

\begin{example}\label{exN}
 Let ${\rm N}=\langle e_1,e_2 : e_1\circ e_1=e_2\rangle$. Every skew-symmetric anti-biderivation $\delta$ of ${\rm N}$ is zero:
\begin{center} $\delta(e_1,e_1)=\delta(e_2,e_2)=0$ and $\delta(e_1,e_2)=\delta(e_1,e_1\circ e_1)=-2\,\delta(e_1,e_1)\circ e_1=0$.\end{center}
In particular, ${\rm N}$ admits neither non-trivial ${\mathcal A}{\rm BD}$-structures nor non-trivial PJJ-structures, although it is a non-trivial Jacobi-Jordan algebra.
This shows that a non-trivial Jacobi-Jordan algebra need not admit a non-trivial PJJ-structure; note also that the map $(x,y)\mapsto\alpha_1\beta_1(\beta_1-\alpha_1){\bf c}$, where $x=\alpha_1e_1+\alpha_2e_2$, $y=\beta_1e_1+\beta_2e_2$ and ${\bf c}\in{\rm N}$ is fixed, which might look like a natural candidate for such a structure, is not bilinear.
\end{example}

\begin{example}\label{exA}
Let ${\rm A}=\langle e_1,e_2\rangle$ be the $2$-dimensional algebra with zero multiplication and let $\star$ be a PJJ-structure on ${\rm A}$, $v=e_1\star e_2=v_1e_1+v_2e_2$.
By Proposition \ref{52}, 
\begin{center}$x\star(y\star z)+y\star(x\star z)=0$;\end{center} for $x=y=e_1,z=e_2$ and $x=y=e_2,z=e_1$ we get $e_1\star v=v_2\, v=0$ and $e_2\star v=-v_1\, v=0$, hence $v=0$ and $\star=0$.
\end{example}

\begin{theorem}\label{Pronew}
Let $({\rm A},\circ)$ be a non-zero finite-dimensional Jacobi-Jordan algebra over a field of characteristic different from $2$ and $3$.
\begin{enumerate}
\item[{\rm (1)}] Let $\omega$ be a skew-symmetric bilinear form on ${\rm A}$ such that $\omega({\rm A},{\rm A}\circ {\rm A})=0$ and let ${\bf c}\in{\rm Ann}({\rm A})$.
Then $x\star y=\omega(x,y)\,{\bf c}$ is an ${\mathcal A}{\rm BD}$-structure on ${\rm A}$. If, moreover, $\omega({\rm A},{\bf c})=0$ $\big($e.g. ${\bf c}\in {\rm A}\circ{\rm A}\big)$, then $\star$ is a PJJ-structure.
\item[{\rm (2)}] ${\rm A}$ admits a non-trivial ${\mathcal A}{\rm BD}$-structure if and only if $\dim {\rm A}/({\rm A}\circ{\rm A})\geq 2$, i.e. if and only if $\dim{\rm A}\geq2$ and ${\rm A}\not\cong{\rm N}$.
\item[{\rm (3)}] ${\rm A}$ admits a non-trivial PJJ-structure if and only if $\dim {\rm A}\geq 3$.
\end{enumerate}
\end{theorem}
\begin{proof}
{\rm (1)} We have $\omega(x,y\circ z){\bf c}=0$ and $(\omega(x,y){\bf c})\circ z+y\circ(\omega(x,z){\bf c})=0$, since ${\bf c}\in{\rm Ann}({\rm A})$;
the skew-symmetry is clear. If $\omega({\rm A},{\bf c})=0$, then
\begin{center}
$(x\circ y)\star z=-\omega(z,x\circ y){\bf c}=0$ and $x\star(y\star z)=\omega(y,z)\omega(x,{\bf c}){\bf c}=0$,
\end{center}so \eqref{Eq1..} holds.

{\rm (2)} Finite-dimensional Jacobi-Jordan algebras are nilpotent \cite{Burd}; hence ${\rm A}^2={\rm A}\circ {\rm A}\neq {\rm A}$ and ${\rm Ann}({\rm A})\neq0$.
If $\dim {\rm A}/{\rm A}^2\geq2$, take a non-zero skew-symmetric form on ${\rm A}/{\rm A}^2$, pull it back to $\omega$ on ${\rm A}$ and take $0\neq{\bf c}\in{\rm Ann}({\rm A})$; by {\rm (1)} we obtain a non-trivial ${\mathcal A}{\rm BD}$-structure.
If $\dim{\rm A}/{\rm A}^2=1$, then ${\rm A}$ is generated by one element $e$; since $3\, e\circ(e\circ e)=0$ by the Jacobi identity, ${\rm A}=\langle e,e\circ e\rangle$, so ${\rm A}$ is one-dimensional
$\big($and every skew-symmetric map vanishes$\big)$ or ${\rm A}\cong {\rm N}$, which has no non-trivial ${\mathcal A}{\rm BD}$-structures by Example \ref{exN}.

{\rm (3)} Let $\dim{\rm A}\geq3$. If ${\rm A}^2=0$, take a basis $e_1,\ldots,e_n$, $\omega=e_1^*\wedge e_2^*$ and ${\bf c}=e_3$. If ${\rm A}^2\neq0$, then $\dim{\rm A}/{\rm A}^2\geq2$ by the proof of {\rm (2)},
and ${\rm Ann}({\rm A})\cap{\rm A}^2\neq0$ $\big($it contains the last non-zero power ${\rm A}^k$, $k\geq2\big)$; take $\omega$ as in {\rm (2)} and $0\neq{\bf c}\in{\rm Ann}({\rm A})\cap{\rm A}^2$.
In both cases {\rm (1)} gives a non-trivial PJJ-structure. If $\dim {\rm A}\leq 2$, then ${\rm A}$ is one-dimensional, the $2$-dimensional zero algebra or ${\rm N}$, and none of them admits a non-trivial PJJ-structure (Examples \ref{exN} and \label{exA}).
\end{proof}

\begin{definition}
	{ 
		Let $({\rm A},\circ)$ be a Jacobi-Jordan algebra. An  ${\mathcal A}{\rm BD}$-structure on $({\rm A},\circ)$ is a skew-symmetric anti-biderivation $\delta$ of $({\rm A},\circ)$.   				
		
	}
\end{definition}

\begin{proposition}
	{ 		
		Let $({\rm A},\circ)$ be a Jacobi-Jordan algebra with zero multiplication. Then,  ${\mathcal A}{\rm BD}$-structures on $({\rm A}, \circ)$  correspond to anticommutative algebras. 
		
	}
\end{proposition}

The existence of non-trivial ${\mathcal A}{\rm BD}$-structures on finite-dimensional Jacobi-Jordan algebras is described in Theorem \ref{Pronew}~(2).

\section{Examples}

 The classification of two-dimensional algebras over an algebraically closed field was established in \cite{KV19}. We present a classification of two-dimensional ${\mathcal A}{\rm BD}$-algebras, which
can be achieved by direct calculations.

\begin{proposition}
Let $({\mathcal A}, \cdot)$ be a non-trivial $2$-dimensional ${\mathcal A}{\rm BD}$-algebra. Then it is isomorphic to one of the algebras listed below.
\begin{longtable}{llllllllllllllllll}
${\mathcal A}^2_{1}$& $:$&  $ e_1\cdot e_1 = e_2$ \\
${\mathcal A}^2_{2}$&$:$&  $ e_1\cdot e_2 = e_2$ & $ e_2\cdot e_1 = -e_2$
\end{longtable}\noindent
In particular, the variety of $2$-dimensional ${\mathcal A}{\rm BD}$-algebras and
the variety of $2$-dimensional nilalgebras with nilindex $3$ coincide. 
\end{proposition}

The above proposition may create the illusion that all ${\mathcal A}{\rm BD}$-algebras are commutative or anticommutative and that they coincide with nilalgebras of nilindex $3$, but this is not the case. Below,
we give a classification of $3$-dimensional ${\mathcal A}{\rm BD}$-algebras, which provides examples of non-(anti)commutative ${\mathcal A}{\rm BD}$-algebras and distinguishes the varieties of
${\mathcal A}{\rm BD}$-algebras and nilalgebras with nilindex $3$.  
The classification of complex $3$-dimensional nilalgebras was established in \cite{ks25} and 
the classification of complex $3$-dimensional flexible nilalgebras with nilindex $3$ was established in \cite{AABS}. We present a classification of $3$-dimensional ${\mathcal A}{\rm BD}$-algebras, which can be achieved by direct calculations.

\begin{theorem}\label{3cl} 
Let $({\mathcal A}, \cdot)$ be a complex $3$-dimensional ${\mathcal A}{\rm BD}$-algebra. 
Then $({\mathcal A},\cdot)$ is an anticommutative\footnote{They are  given in \cite{ks25}.} or commutative\footnote{They are  given in \cite{Burd}.} algebra,   
or
 it is isomorphic to one of the following algebras:

\begin{longtable}{lllllllll}
${\mathcal A}^3_{1}$ & $:$ & $e_{1}\cdot e_{1} = e_{2}$ &  $e_{1}\cdot e_{3} = e_{3}$  & $e_{3}\cdot e_{1} = -e_{3}$  \\ 
${\mathcal A}^3_{2}$ & $:$ & $e_{1}\cdot e_{1} = e_{2}$ &  $e_{1}\cdot e_{3} = e_{2}$  & $e_{3}\cdot e_{1} = -e_{2}$  \\ 

${\mathcal A}^3_{3}({\alpha\neq0})$ & $:$ & $e_{1}\cdot e_{1} = e_{2}$ & $e_{1}\cdot e_{3} = \alpha e_{2}$  & $e_{3}\cdot e_{1} = - \alpha e_{2}$& $e_{3}\cdot e_{3} = e_{2}$ & \\ 

\end{longtable}
\noindent All listed algebras are non-isomorphic except ${\mathcal A}_{3}^3({\alpha
})\cong {\mathcal A}_{3}^3({-\alpha })$.
In particular, 
\begin{enumerate}
    \item[{\bf A}.] the variety of $3$-dimensional nilpotent nilalgebras with nilindex $3$ and 
the variety of $3$-dimensional nilpotent ${\mathcal A}{\rm BD}$-algebras
coincide$;$
\item[{\bf B}.] the variety of $3$-dimensional flexible nilalgebras with nilindex $3$ and 
the variety of $3$-dimensional ${\mathcal A}{\rm BD}$-algebras
coincide$;$
\item[{\bf C}.] 
up to isomorphism there are exactly three $3$-dimensional non-${\mathcal A}{\rm BD}$ nilalgebras with nilindex $3$ (none of them is nilpotent):
\[\begin{array}{llll}
{\mathcal M}_1: & e_1e_1=e_2, & e_1e_3=e_1, & e_3e_1=-e_1;\\
{\mathcal M}_2: & e_1e_1=e_2,\ e_3e_3=e_2, & e_1e_3=e_1+ie_3, & e_3e_1=-e_1-ie_3;\\
{\mathcal M}_3: & e_1e_1=e_2,\ e_3e_3=e_2, & e_1e_3=e_1, & e_3e_1=-e_1.
\end{array}\]
Here ${\mathcal M}_1^+\cong {\rm N}\oplus{\mathbb C}$, while ${\mathcal M}_2^+\cong{\mathcal M}_3^+$ and the line $[{\mathcal M}_i,{\mathcal M}_i]$ is isotropic for ${\mathcal M}_2$ and non-isotropic for ${\mathcal M}_3$ with respect to the form defining ${\mathcal M}_i^+$.

\end{enumerate}\end{theorem}

The above theorem may create the illusion that
the variety of flexible nilalgebras with nilindex $3$ coincides with
the variety of ${\mathcal A}{\rm BD}$-algebras and that
all nilpotent nilalgebras with nilindex $3$ are ${\mathcal A}{\rm BD}$-algebras, but this is not the case. Below,
we give a classification of nilpotent $4$-dimensional ${\mathcal A}{\rm BD}$-algebras, which shows the difference between these two  varieties.   
The classification of complex $4$-dimensional nilpotent nilalgebras  with nilindex $3$ was established in \cite{AKK}. 
We present a classification of $4$-dimensional nilpotent ${\mathcal A}{\rm BD}$-algebras, which can be achieved by direct calculations.

\begin{proposition}
    The minimal dimension of non-flexible     ${\mathcal A}{\rm BD}$-algebras is $4$.
\end{proposition}
\begin{proof}
    The  algebra 
    \begin{longtable}{llllllll}
${\mathcal A}^4_0$ & $:$& $e_1e_1=e_2$ & $e_1e_3=-e_3+e_4$ & $e_1e_4=e_4$ & 
$e_2e_3=2e_4$\\ & && $e_3e_1=e_3+e_4$ & $e_4e_1=-e_4$ & $e_3e_2=-2e_4$ 
    \end{longtable}
 \noindent   is a $4$-dimensional non-flexible ${\mathcal A}{\rm BD}$-algebra, 
    while, by Theorem \ref{3cl}, each $3$-dimensional ${\mathcal A}{\rm BD}$-algebra is flexible.
\end{proof}

\begin{theorem}\label{4ncl} 
Let $({\mathcal A},\cdot)$ be a complex $4$-dimensional nilpotent ${\mathcal A}{\rm BD}$-algebra. 
Then $({\mathcal A},\cdot)$ is a $2$-step nilpotent algebra or it is isomorphic to one of the algebras from the following list:

\begin{longtable}{llllllllllllll}

$\mathcal{A}^4_{1}$ & $: $&  $e_1\cdot e_2 = e_3$ & $ e_1\cdot e_3=e_4$ 
&  $e_2\cdot e_1 =- e_3$ & $ e_3\cdot e_1=-e_4$\\ 
 
$\mathcal{A}^4_{2}$  & $: $& $e_1\cdot e_2=-e_3+e_4$ & $e_1\cdot e_3=-e_4$ &  $e_2\cdot e_1 = e_3$ &   $e_3\cdot e_1=e_4$ \\

$\mathcal{A}^4_{3}$ & $: $& $e_1\cdot e_2=-e_3$ & $e_1\cdot e_3=-e_4$ &  $e_2\cdot e_1 = e_3$ &   $e_2\cdot e_2=e_4$  & $e_3\cdot e_1=e_4$   \\

$\mathcal{A}^4_{4}$ & $: $ & $e_1\cdot e_1 = e_4$ &  $e_1\cdot e_2=-e_3$ & $e_1\cdot e_3=-e_4$\\& &  $e_2\cdot e_1 = e_3$ & $e_2\cdot e_2=e_4$ & $e_3\cdot e_1=e_4$ \\ 
 
$\mathcal{A}^4_{5}$ & $: $& $e_1\cdot e_1 = e_4$ & $e_1\cdot e_2=-e_3$ & $e_1\cdot e_3=-e_4$ & $e_2\cdot e_1 = e_3$ &   $e_3\cdot e_1=e_4$ \\ 
\end{longtable}
\noindent 
In particular, 
\begin{enumerate}
    \item[{\bf A.}] the varieties of $4$-dimensional nilpotent symmetric Leibniz algebras 
and $4$-dimensional nilpotent ${\mathcal A}{\rm BD}$-algebras coincide$;$
\item[{\bf B.}]  there are infinitely many $4$-dimensional non-${\mathcal A}{\rm BD}$ nilpotent nilalgebras with nilindex $3,$ listed in {\rm \cite[Theorem C]{AKK}} as $\mathbb{M}_i.$ 

\end{enumerate}
\end{theorem}

\begin{proposition}
    The minimal dimension of 
    nilpotent (non-symmetric Leibniz) ${\mathcal A}{\rm BD}$-algebras is $5.$
\end{proposition}

\begin{proof}
    By Theorem \ref{4ncl}, the varieties of $4$-dimensional nilpotent symmetric Leibniz algebras and nilpotent ${\mathcal A}{\rm BD}$-algebras coincide. 
The following algebra gives an example of a $5$-dimensional 
non-symmetric Leibniz ${\mathcal A}{\rm BD}$-algebra: 
    
        \begin{longtable}{llllllll}
${\mathcal A}^5_0$ & $:$& $e_1\cdot e_2=e_3$ & $e_2\cdot e_1=-e_3$ & $e_2\cdot e_3=e_4$ & 
$e_3\cdot e_2=-e_4$ & $e_3\cdot e_4=e_5$ & $e_4\cdot e_3=-e_5.$      \end{longtable}
\end{proof}

\begin{definition}[see \cite{MS,FKS}]
  An algebra $(\rm A, \cdot)$ is called rigid in a variety $\mathcal{V}$ if the Zariski closure of its orbit under the action of ${\rm GL}(\rm A)$    is an open subset of $\mathcal{V}$. A subset of a variety is called irreducible if it cannot be represented as a union of two non-trivial closed subsets.
 A maximal irreducible closed subset of a variety is called an  irreducible component.

\end{definition}

\begin{corollary}
    The variety of complex $3$-dimensional  ${\mathcal A}{\rm BD}$-algebras has dimension $9$ and  
    $3$ irreducible components, defined by
    the generic family for anticommutative algebras, 
    the family of algebras  ${\mathcal A}^3_{3}({\alpha})$ and one rigid algebra ${\mathcal A}^3_{1}.$
\end{corollary}

\begin{proof}
    The complete graph of degenerations of complex $3$-dimensional nilalgebras is given in \cite{ks25}. Taking the subgraph which corresponds to ${\mathcal A}{\rm BD}$-algebras, we have our statement.
\end{proof}

\begin{corollary}[see \cite{AK21}]
 The variety of complex  $4$-dimensional nilpotent  ${\mathcal A}{\rm BD}$-algebras has dimension $12$ and three irreducible components, defined by 
 one rigid algebra $\mathcal{A}_4^4$ and 
 two components corresponding to $2$-step nilpotent algebras.
Indeed, $\dim\mathfrak{Der}(\mathcal{A}^4_4)=4$, so the orbit of $\mathcal{A}^4_4$ has dimension $16-4=12$; the value $11$ stated in {\rm \cite[Theorem D]{AK21}} has to be corrected.

\end{corollary}

Jacobson proved that each complex finite-dimensional Lie algebra with an invertible derivation is nilpotent \cite{J55}.
The minimal dimension of complex nilpotent Lie algebras without invertible derivations is $7$  $\big($an example of a $7$-dimensional Lie algebra without invertible derivations is given in \cite{B02} and by direct calculation of derivations of nilpotent $6$-dimensional Lie algebras from \cite{S}, we can see that all of them admit invertible derivations$\big)$.

\begin{proposition}
    The minimal dimension of nilpotent  ${\mathcal A}{\rm BD}$-algebras without invertible derivations is $4.$
\end{proposition}

\begin{proof}
By direct checking of derivations of $3$-dimensional nilpotent  ${\mathcal A}{\rm BD}$-algebras, we see that all of them  admit invertible derivations.
The derivations of   $\mathcal{A}^4_{2}$ have the following form
\begin{longtable}{lcllcl}
$\varphi(e_1)$ &$=$&$\alpha_{13}e_3+\alpha_{14}e_4,$ & 
$\varphi(e_2)$ &$=$&$\alpha_{33}e_2+\alpha_{34}e_3+\alpha_{24}e_4,$ \\
$\varphi(e_3)$ &$=$&$\alpha_{33}e_3+\alpha_{34}e_4,$ & 
$\varphi(e_4)$ &$=$&$\alpha_{33} e_4.$ 
\end{longtable}\noindent
Hence,  each derivation of   $\mathcal{A}^4_{2}$ is non-invertible.
\end{proof}

We finish our paper with the answer to the following question motivated by Theorem \ref{3cl}:
{\it is a perfect $\big($in particular, simple$\big)$ ${\mathcal A}{\rm BD}$-algebra anticommutative?}
The answer is positive for simple ${\mathcal A}{\rm BD}$-algebras $\big($Theorem \ref{simple}; in particular, for all complex finite-dimensional ones$\big)$
and negative for perfect ${\mathcal A}{\rm BD}$-algebras $\big($Theorem \ref{perfect}$\big)$.

\begin{theorem}\label{simple}
Let $({\rm A},\cdot)$ be a simple ${\mathcal A}{\rm BD}$-algebra such that ${\rm A}\circ{\rm A}\neq{\rm A}$ $\big($this holds for every finite-dimensional ${\mathcal A}{\rm BD}$-algebra over a field of characteristic different from $2$ and $3$, since ${\rm A}^+$ is then a nilpotent Jacobi-Jordan algebra \cite{Burd}$\big)$.
Then $({\rm A},\cdot)$ is anticommutative.
\end{theorem}
\begin{proof}
By Lemma \ref{bdnov} with $I=J={\rm A}$, the subspace ${\rm A}\circ{\rm A}$ is a two-sided ideal of $({\rm A},\cdot)$. As ${\rm A}$ is simple and ${\rm A}\circ{\rm A}\neq{\rm A}$, we get ${\rm A}\circ{\rm A}=0$.
\end{proof}

\begin{theorem}\label{perfect}
Let $(V,\nu)$ be an anticommutative algebra and $\kappa$ a symmetric bilinear form on $V$ which is invariant: $\kappa(\nu(x,y),z)+\kappa(y,\nu(x,z))=0$.
On ${\rm A}=V\oplus{\mathbb C}$ define $(x,a)\cdot(y,b)=\big(\nu(x,y),\kappa(x,y)\big)$. Then $({\rm A},\cdot)$ is an ${\mathcal A}{\rm BD}$-algebra; it is not anticommutative if $\kappa\neq0$,
and it is perfect if $(V,\nu)$ is perfect and $\kappa\neq0$.
In particular, the $4$-dimensional algebra
\begin{longtable}{lcl}
${\mathcal P}_4$&$:$&
$e_1e_2=e_3,\ e_2e_3=e_1,\ e_3e_1=e_2,\ e_2e_1=-e_3,$\\ &&$e_3e_2=-e_1,\ e_1e_3=-e_2,\ e_1e_1=e_2e_2=e_3e_3=e_4$
\end{longtable}\noindent 
 is a perfect non-anticommutative ${\mathcal A}{\rm BD}$-algebra.
\end{theorem}
\begin{proof}
We have $X\circ Y=(0,2\kappa(x,y))$ for $X=(x,a)$, $Y=(y,b)$, so 
\begin{center}
${\rm A}\circ{\rm A}\subseteq 0\oplus{\mathbb C}$ and $({\rm A}\circ{\rm A})\cdot{\rm A}={\rm A}\cdot({\rm A}\circ{\rm A})=0$.\end{center}
Hence the left hand sides of \eqref{Def1} and \eqref{Def2} vanish, and the right hand side of \eqref{Def1} equals \begin{center}
$-2\big(0,\kappa(\nu(x,y),z)+\kappa(y,\nu(x,z))\big)=0$;
\end{center} \eqref{Def2} is analogous.
If $V=\nu(V,V)$ and $\kappa(x_0,y_0)\neq0$, then ${\rm A}\cdot{\rm A}$ contains $0\oplus{\mathbb C}$ $\big($as $X_0\cdot Y_0+Y_0\cdot X_0=(0,2\kappa(x_0,y_0))\big)$ and all elements $(\nu(x,y),\kappa(x,y))$, hence ${\rm A}\cdot{\rm A}={\rm A}$.
\end{proof}

The algebras of Theorem \ref{perfect} are exactly the ${\mathcal A}{\rm BD}$-algebras whose plus square ${\rm A}\circ{\rm A}$ is one-dimensional.

\begin{theorem}\label{circone}
Let $({\rm A},\cdot)$ be an algebra over a field ${\mathbb F}$ of characteristic different from $2$ and $3$, and assume that
$x\circ y=\kappa(x,y)e$ for a nonzero element $e\in{\rm A}$ and a nonzero bilinear form $\kappa$ on ${\rm A}$ $\big($i.e. $\dim {\rm A}\circ{\rm A}=1\big)$.
Then $({\rm A},\cdot)$ is an ${\mathcal A}{\rm BD}$-algebra if and only if
$e\cdot{\rm A}={\rm A}\cdot e=0$ and $\kappa(x\cdot y,z)+\kappa(y,x\cdot z)=0$ for all $x,y,z\in{\rm A}$.
\end{theorem}
\begin{proof}
The form $\kappa$ is symmetric, since $\circ$ is commutative.
Let $({\rm A},\cdot)$ be an ${\mathcal A}{\rm BD}$-algebra.
The Jacobi identity of the Jacobi-Jordan algebra ${\rm A}^+$ $\big($Proposition \ref{ABDJJA}$\big)$ applied to $e,e,e$ gives $3\kappa(e,e)^2e=0$, i.e. $\kappa(e,e)=0$;
applied to $x,x,e$ it gives $2\kappa(x,e)^2e=0$, i.e. $\kappa(x,e)=0$ for all $x$.
Now \eqref{Def1} reads
\begin{equation}\label{circone1}
\big(\kappa(x\cdot y,z)+\kappa(y,x\cdot z)\big)e\ =\ -\kappa(y,z)\, x\cdot e.
\end{equation}
Since $\kappa(x\cdot y+y\cdot x,z)=\kappa(x,y)\kappa(e,z)=0$, the sum of the six equalities \eqref{circone1} for all permutations of $x,y,z$ gives\begin{center}
$2\big(\kappa(y,z)\,x\cdot e+\kappa(z,x)\,y\cdot e+\kappa(x,y)\,z\cdot e\big)=0.$\end{center}
As $\kappa\neq0$ is symmetric and ${\rm char}\,{\mathbb F}\neq2$, there is $w$ with $\kappa(w,w)\neq0$.
Taking $x=y=z=w$ we get $3\kappa(w,w)\,w\cdot e=0$, so $w\cdot e=0$; taking $y=z=w$ we get $\kappa(w,w)\,x\cdot e=0$, so ${\rm A}\cdot e=0$.
Then $e\cdot x=\kappa(e,x)e-x\cdot e=0$, and \eqref{circone1} gives the invariance of $\kappa$.
Conversely, if $e$ annihilates ${\rm A}$ and $\kappa$ is invariant, then the left hand sides of \eqref{Def1} and \eqref{Def2} vanish, the right hand side of \eqref{Def1} equals
$-\big(\kappa(x\cdot y,z)+\kappa(y,x\cdot z)\big)e=0$, and the right hand side of \eqref{Def2} equals \begin{center} $-\big(\kappa(y\cdot x,z)+\kappa(y,z\cdot x)\big)e=\big(\kappa(x\cdot y,z)+\kappa(y,x\cdot z)\big)e=0$\end{center}
$\big($we used $y\cdot x=\kappa(x,y)e-x\cdot y$ and $\kappa(e,\cdot)=0\big)$.
\end{proof}

\begin{corollary}\label{perfmin}
Let $({\rm A},\cdot)$ be a finite-dimensional perfect non-anticommutative ${\mathcal A}{\rm BD}$-algebra over a field of characteristic different from $2$ and $3$.
Then $\dim{\rm A}\geq4$. If $\dim{\rm A}=4$, then $({\rm A},\cdot)$ is isomorphic to the algebra $V\oplus{\mathbb F}$ defined as in Theorem \ref{perfect},
where $(V,\nu)$ is a $3$-dimensional perfect anticommutative algebra and $\kappa\neq0$ is a symmetric invariant bilinear form on $V$.
In particular, ${\mathcal P}_4$ is a perfect non-anticommutative ${\mathcal A}{\rm BD}$-algebra of minimal dimension.
\end{corollary}
\begin{proof}
By Lemma \ref{bdnov}, $I={\rm A}\circ{\rm A}\neq0$ is an ideal of $({\rm A},\cdot)$, and $I\neq{\rm A}$, since ${\rm A}^+$ is a nilpotent Jacobi-Jordan algebra \cite{Burd}.
The quotient ${\rm A}/I$ is a nonzero perfect anticommutative algebra. An anticommutative algebra of dimension $n\leq 2$ satisfies $\dim {\rm A}\cdot{\rm A}\leq n(n-1)/2<n$,
hence $\dim{\rm A}/I\geq3$ and $\dim{\rm A}\geq4$.
Let $\dim{\rm A}=4$. Then $I={\mathbb F}e$, and by Theorem \ref{circone}, $x\circ y=\kappa(x,y)e$, where $e$ annihilates ${\rm A}$ and $\kappa\neq0$ is invariant.
By the proof of Theorem \ref{circone}, $\kappa({\rm A},e)=0$, so $\kappa$ induces a symmetric invariant form on $V={\rm A}/I$, and $V$ is a $3$-dimensional perfect anticommutative algebra.
Choosing a complement of ${\mathbb F}e$ identified with $V$, we get \begin{center} $x\cdot y=\nu(x,y)+\big(\frac12\kappa(x,y)+\omega(x,y)\big)e$ for $x,y\in V$ with a skew-symmetric form $\omega$.\end{center}
Replacing the complement by the graph $\{x+f(x)e\}$ of a linear form $f$ replaces $\omega$ by $\omega-f\circ\nu$.
Since $V$ is perfect and $\dim V=\dim\Lambda^2V=3$, the map $\nu:\Lambda^2V\to V$ is bijective, so we can choose $f$ with $f\circ\nu=\omega$.
Rescaling $e$ by $\frac12$, we obtain the algebra of Theorem \ref{perfect}.
\end{proof}

The following simple proposition gives a sufficient condition for an ${\mathcal A}{\rm BD}$-algebra not to  be perfect. 

\begin{proposition}\label{prop54}
    Let $({\rm J}, \circ)$ be a  Jacobi-Jordan algebra, which decomposes into the following direct sum $\big($of vector spaces$\big)$ 
    ${\rm J}={\mathcal J}+{\rm J}^2$ and
     which does not admit a skew-symmetric anti-biderivation $\mathfrak D$ with the condition  
    ${\mathcal J} \subseteq \mathfrak  D({\rm J},{\rm J})+{\rm J}^2.$
Then each ${\mathcal A}{\rm BD}$-algebra $({\rm A}, \cdot),$ such that ${\rm A}^+ = {\rm J},$ is not perfect.
$\big($Indeed, ${\rm A}\cdot{\rm A}=[{\rm A},{\rm A}]+{\rm A}\circ{\rm A}$ and $[\cdot,\cdot]$ is a skew-symmetric anti-biderivation of ${\rm J}$ by Proposition \ref{pr2}.$\big)$

\end{proposition}

The following result follows directly from the main theorem of \cite{Block68}.

\begin{proposition}
    Let $({\rm A}, \cdot)$ be a complex finite-dimensional simple flexible ${\mathcal A}{\rm BD}$-algebra. Then $({\rm A}, \cdot)$ is anticommutative.
$\big($By Theorem \ref{simple}, the flexibility assumption is not needed.$\big)$
\end{proposition}

 Let us also mention that the above question is related to a non-commutative version of Albert's problem: {\it is every finite-dimensional $\big($commutative$\big)$ power
associative nilalgebra solvable}? 
For each $n > 3,$ Correa and Hentzel constructed  non-(anti)commutative $n$-dimensional non-solvable nilalgebras \cite{CH}, 
but these algebras are not ${\mathcal A}{\rm BD}$-algebras.

\bigskip 




  %


\end{document}